\documentclass[a4paper,11pt]{amsart}
\usepackage{amsthm,amssymb,mathtools,amsfonts,latexsym,rawfonts,amsmath, mathrsfs, lscape}
\usepackage{stmaryrd,tikz,pgfplots}
\usepackage{marginnote}  
\usepackage[backref=page]{hyperref}
\usepackage[top=1.5in, bottom=0.9in, left=0.9in, right=0.9in]{geometry}
\usepackage{comment}
\usepackage{enumitem}

\newlist{legal}{enumerate}{10}
\setlist[legal]{label*=\arabic*.}
\usepackage[textsize=tiny]{todonotes}

\newcommand{\Marginpar}[1]{\marginpar{\tiny{#1}}}
\newcommand{\Note}[1]{{\par\noindent\hrulefill\par\tiny{#1}\par\noindent\hrulefill\par}}
\newcommand{\Detail}[1]{{#1}}
\renewcommand{\Marginpar}[1]{}
\renewcommand{\Note}[1]{}
\renewcommand{\Detail}[1]{}
\renewcommand*{\backref}[1]{}
\renewcommand*{\backrefalt}[4]{%
    \ifcase #1 (Not cited.)%
    \or        (Cited on page~#2.)%
    \else      (Cited on pages~#2.)%
    \fi}

\usepackage{hyperref}
\usepackage[all]{xy}
\usepackage{amsmath,amsfonts}	% Bunch of math support
\usepackage{amsthm,latexsym,amssymb}
\usepackage{tensor}

\usepackage{array, tabularx}

\usepackage{booktabs} %for nice tables
\usepackage{siunitx}

\usepackage{tikz-cd}

\newtheorem{thm}{Theorem}
\newtheorem*{thm*}{Theorem}
\newtheorem{prop}[thm]{Proposition}
\newtheorem{lem}[thm]{Lemma}
\newtheorem*{lem*}{Lemma}
\newtheorem{cor}[thm]{Corollary}
\newtheorem*{cor*}{Corollary}

\theoremstyle{definition}
\newtheorem{defn}[thm]{Definition}
\newtheorem*{defn*}{Definition}

\newtheorem{rem}[thm]{Remark}

\renewcommand{\[}{\begin{equation*}}
\renewcommand{\]}{\end{equation*}}

\DeclareMathOperator\tr{tr}

\def \R {\mathbb R}

\def \g {\mathfrak g}

\def \n {\mathfrak n}

\def \tr {\mbox{tr}}
\def \ad {\mbox{ad}}

\begin{document}
\parskip 1mm

\title[Pluricanonical LCAK metrics]{On pluricanonical locally conformally almost K\"ahler metrics}

\author{Ethan Addison}
\address{Department of Mathematics, Texas State University, San Marcos, TX 78666-4684, USA}
\email{el.addison@txstate.edu}

\author{Tedi Dr\u{a}ghici} 
\address{Department of Mathematics and Statistics, Florida International University, Miami, FL 33199, USA}
\email{draghici@fiu.edu}

\author{Mehdi Lejmi}
\address{Department of Mathematics, Bronx Community College of CUNY, Bronx, NY 10453, USA.}
\email{mehdi.lejmi@bcc.cuny.edu}

\thanks{ }



\maketitle
\begin{abstract}
On an almost complex manifold $(M,J)$, a pluricanonical locally conformally almost K\"ahler (LCAK)  metric $g$ is induced by a locally conformally symplectic structure $(F,\theta)$ of the first kind, characterized by the fact that $D\theta$ is $J$-anti-invariant and that the image of the Nijenhuis tensor is $g$-orthogonal to the distribution spanned by $\{\theta^\sharp,J\theta^\sharp\}$, where $\theta$ is the Lee form and $D$ is the Levi-Civita connection. On a compact complex manifold, pluricanonical locally conformally K\"ahler (LCK) metrics have parallel Lee form. The same conclusion holds for LCK Chern--Ricci flat Gauduchon metrics. We generalize both results to LCAK metrics.
We also observe that on a compact pluricanonical LCAK manifold with a non-trivial Lee form, there is no symplectic form compatible with the same almost complex structure. Moreover, we remark that the pluricanonical LCAK condition implies that the fundamental $2$-form is an eigenform of the Hodge Laplacian, and we give a simple characterization of the pluricanonical LCAK condition on compact manifolds. Finally, we study LCAK metrics with $\theta^\sharp$ being real holomorphic, proving in that case $D\theta=0$ when the metric is Gauduchon.
\end{abstract}
  
  \section{Introduction}
  A locally conformally symplectic (LCS for short) structure on a manifold of real dimension $2n$ is given by a pair $(F,\theta)$, where $F$ is a non-degenerate $2$-form and $\theta$ is a closed $1$-form such that $dF=\theta\wedge F.$ In particular, $\theta$ is $d$-exact precisely when $F$ is globally conformal to a symplectic form. For further details about LCS structures, we refer the reader for instance to~\cite{MR1481969,MR4771164} or the excellent survey of Bazzoni~\cite{MR3880223}. The $2$-form $F$ is said to be compatible with an almost complex structure $J$ if $g:=F(\cdot,J\cdot)$ defines a Riemannian metric. The induced metric $g$ is called a locally conformally almost K\"ahler (LCAK for short) metric. If $J$ is integrable, then $g$ is a locally conformally K\"ahler (LCK for short) metric. Moreover, when $\theta$ is parallel with respect to the Levi-Civita connection $D$ of $g$, the LCK metric $g$ is called Vaisman. We remark that an LCS structure giving rise to a (non-K\"ahler) Vaisman metric is of the first kind~\cite{MR809073}, namely there exists a vector field $T$ such that both $\mathcal{L}_TF=0$ and $\theta(T)=1$ (here $\mathcal{L}$ is the Lie derivative). Moreover, the integrable almost complex structure inducing the Vaisman metric is adapted to the LCS structure (see Definition~\ref{definition-adapted}).

  It is natural to consider LCS structures $(F,\theta)$ of the first kind on almost complex manifolds such that the almost complex structure (not necessarily integrable) $J$ is adapted to the LCS structure. We call the induced metric $g$ a pluricanonical LCAK metric (see Definition~\ref{definition-pluricanonical}). Pluricanonical LCAK metrics are characterized by the fact that the image of the Nijenhuis tensor is $g$-orthogonal to the distribution spanned by $\{\theta^\sharp,J\theta^\sharp\}$ and that $D\theta$ is $J$-anti-invariant (see Corollary~\ref{equivalence-first_kind-pluri}). We point out that, when $J$ is integrable, pluricanonical LCK metrics were introduced by Kokarev~\cite{MR2520354} in the context of harmonic maps. Ornea and Verbitsky~\cite{MR4771164} and A. Moroianu and S. Moroianu~\cite{MR3674175} proved that, on compact complex manifolds, pluricanonical LCK metrics are in fact Vaisman. However, in the non-integrable case, the Lee form of a pluricanonical LCAK metric is not necessarily $D$-parallel, even in the compact case (see examples in Section~\ref{Sec3}). It turns out that there is a condition involving the Nijenhuis tensor of $J$, stated in the following theorem, in order for $\theta$ to be $D$-parallel.

\begin{thm*}(Theorem~\ref{generalization-Ornea-Verbitsky})
Let $(M, g, J, F, \theta)$ be a compact manifold of real dimension $2n$ with a pluricanonical LCAK structure. Suppose that 
\begin{equation}\label{theta-killing-condition-intro}
\int_Mg(N(T),D\theta)\,F^n=0,
\end{equation}
where $N$ is the Nijenhuis tensor, $T=\theta^\sharp$, and $N(T)=g\big(N(T,\cdot),\cdot \big)$. Then $\theta$ is $D$-parallel.
\end{thm*}
We point out that our proof is different from~\cite{MR4771164,MR3674175} and is actually based on a Bochner-type formula. When $J$ is integrable, the condition~(\ref{theta-killing-condition-intro}) is trivially satisfied and so we recover the result in~\cite{MR4771164,MR3674175}.

Moreover, it was proved in~\cite{angella2023notecompatibilityspecialhermitian} that a compact complex manifold equipped with a Vaisman metric $g$ does not admit a balanced metric (i.e. $d\omega^{n-1}=0$, where $\omega$ is the non-degenerate $2$-form induced by the metric) unless the metric $g$ is K\"ahler. A similar conclusion holds for pluricanonical LCAK metrics. This is in analogy with Vaisman’s theorem~\cite{Vais}, which establishes the mutual exclusivity of LCK not globally K\"ahler metrics and Kähler metrics on the same complex manifold (see also~\cite{MR4824944}).

\begin{thm*} (Theorem~\ref{main-thm-1})
 Let $(M,F,\theta)$ be a compact manifold of real dimension $2n$ with an LCS structure of the first kind. Let $J$ be an adapted almost complex structure for $(F, \theta)$. Then, there is no non-degenerate $2$-form $\omega$ compatible with $J$ such that $d\omega^{n-1}=0.$ In particular, there is no symplectic form $\omega$ compatible with $J$. 
 \end{thm*}
We also observe that when the metric is pluricanonical LCAK, the associated fundamental $2$-form is an eigenform of the Hodge Laplacian. Additionally, in~\cite{MR3397500} it was shown that an LCK metric on a unimodular solvable Lie algebra is Vaisman if and only if $g([T,JT],JT)=0,$ where $T=\theta^\sharp.$ We extend that result to pluricanonical LCAK metrics on compact manifolds.

\begin{thm*} (Theorem~\ref{compchars-plurican})
 Let $(M, g, J, F, \theta)$ be a compact manifold of dimension $2n$ with an LCAK structure that is not almost K\"ahler (i.e. $\theta$ is not identically zero). The following statements are equivalent:

 (i) The structure $(g, J, F, \theta)$ is pluricanonical; 

 (ii) $g$ is Gauduchon, $T$ is orthogonal to the image of the Nijenhuis tensor, and $J\theta([T,JT]) = 0$;

 (iii) The trace-free part of $D \theta$ is $J$-anti-invariant and 
 $\| \theta \|$ is a constant.

\noindent Moreover, if $2n \geq 6$, the above are also equivalent with

(iv) $F$ is an eigenform of the Hodge Laplacian.
\end{thm*}

\begin{comment}
We also observe that when the metric is pluricanonical LCAK, the associated fundamental $2$-form is an eigenform of the Hodge Laplacian.
\begin{cor*}(Corollary~\ref{eigenform})
Let $(M, g, J, F, \theta)$ be a compact manifold of real dimension $2n$ with a Gauduchon LCAK structure. In dimension $2n \geq 6$, the metric $g$ is pluricanonical LCAK if and only if $F$ is an eigenform of the Hodge Laplacian. In dimension $2n=4$, the pluricanonical condition on $g$ implies the eigenform property of $F.$
\end{cor*}

Additionally, in~\cite{MR3397500} it was shown that an LCK metric on a unimodular solvable Lie algebra is Vaisman if and only if $g([T,JT],JT)=0,$ where $T=\theta^\sharp.$ We extend that result to pluricanonical LCAK metrics on compact manifolds.

\begin{cor*}(Corollary~\ref{pluricanonical-condition})
Let $(M, g, J, F, \theta)$ be a compact manifold of real dimension $2n$ with a Gauduchon LCAK structure and assume that $T$ is orthogonal to the image of the Nijenhuis tensor. Then the metric $g$ is pluricanonical LCAK if and only if $J\theta\left([T,JT]\right) = 0.$
\end{cor*}
\end{comment}

Furthermore, in a recent paper~\cite{barbaro2025calabiyaulocallyconformallykahler}, it was shown that on a compact complex manifold, LCK metrics which are Hermitian-Ricci flat  with respect to some canonical Hermitian connection in the $1$-parameter family of connections $\nabla^t$ introduced by Gauduchon~\cite{MR1456265} are in fact Vaisman provided that the metric is Gauduchon. It turns out that this can be extended to LCAK metrics.

\begin{thm*}(Theorem~\ref{Gauduchon-Ricci-flat})
Let $(M, g, J, F, \theta)$ be a compact manifold of real dimension $2n$ with a Gauduchon LCAK structure such that $T=\theta^\sharp$ is $g$-orthogonal to the image of the Nijenhuis tensor. Suppose also that the structure is Hermitian-Ricci flat with respect to some canonical Hermitian connection $\nabla^t$, and that $$\int_Mg(N(T),D\theta)\,F^n=0.$$ Then $\theta$ is $D$-parallel.
%we have the following:
%$\begin{enumerate}[label=\arabic*)]
%$\item if $\int_Mg([T,JT],JT)\,F^n=0$, then $g$ is a pluricanonical LCAK metric.\\
%\item if $\int_Mg(N(T),D\theta)\,F^n=0$, then $\theta$ is $D$-parallel. 
%\end{enumerate}
\end{thm*}

Finally, we study LCAK metrics $g$ when the image of the Nijenhuis tensor is $g$-orthogonal to Span$(T,JT)$ and $D\theta$ is $J$-invariant (as opposed to pluricanonical LCAK metrics). We call such LCAK metrics {\it anti-pluricanonical} LCAK (see Definition~\ref{def-anti-pluricanonical}). We prove that the anti-pluricanonical condition is equivalent to the fact that $T=\theta^\sharp$ is a real holomorphic vector field (see Proposition~\ref{holomorphic}). It was actually shown by A. Moroianu, S. Moroianu and Ornea~\cite{MR3830777} that on a compact LCK manifold, if $\theta^\sharp$ is a holomorphic vector field, then the LCK metric $g$ is Vaisman provided that $g$ is Gauduchon. The following extends their result to the non-integrable setting. 

 \begin{thm*}(Theorem~\ref{anti-pluricanonical})
Let $(M, g, J, F, \theta)$ be a compact manifold of real dimension $2n$ with an LCAK structure.
 Suppose that the image of the Nijenhuis tensor is $g$-orthogonal to Span$(T,JT)$.
 Then, $g$ is an anti-pluricanonical LCAK Gauduchon metric if and only if $\theta$ is $D$-parallel.
 \end{thm*}

The paper is organized as follows: in Section~\ref{Sec1}, we first discuss the link between locally conformally symplectic structures of the first kind and the pluricanonical LCAK condition. Then, we prove several results about pluricanonical LCAK metrics. In Section~\ref{sec-anti-plu}, we introduce and study anti-pluricanonical LCAK metrics.
In Section~\ref{Sec3}, we provide compact examples of pluricanonical (non-LCK) LCAK metrics with both a non $D$-parallel and $D$-parallel Lee form, and we classify unimodular almost abelian pluricanonical LCAK Lie algebras in dimension $4.$

  \section{Pluricanonical locally conformally almost K\"ahler metrics}\label{Sec1}

We start by fixing some notations. Throughout the paper, a non-degenerate $2$-form $\omega$ on an almost complex manifold $(M,J)$ of real dimension $2n$ is said to be compatible with the almost complex structure $J$ if $g:=\omega(\cdot,J\cdot)$ is a Riemannian metric. The Riemannian metric $g$ is called then an almost Hermitian metric. The square of the norm of the metric $g$ is $\|g\|^2=2n$ while $\|\omega\|^{2}=n.$ We denote by $D$ the Levi-Civita connection of $g$. Let $R$ be the Riemannian curvature of $D$ with the convention $R_{X,Y}=D_{[X,Y]}-[D_X,D_Y]$, for $X$ and $Y$ vector fields. We define then the star-Ricci form as $\rho^\ast=R(\omega)$, where $R$ is seen as an endomorphism of $2$-forms. The Nijenhuis tensor of the almost complex structure $J$ is given by (note the convention)
$$4N(X,Y)=[JX,JY]-[X,Y]-J[JX,Y]-J[X,JY].$$ In particular, if a vector field $X$ is orthogonal to the image of the Nijenhuis tensor $N$, then $JX$ is also orthogonal. We denote by $N_X$ the $2$-form $N_X:=g(N(\cdot,\cdot),X)$ and by $N(X)$ the $2$-tensor $N(X):=g(N(X,\cdot),\cdot).$ 

For any $2$-tensor $\phi$, we denote its $J$-invariant part by 
$$\phi^{J,+}(X,Y)=\frac{1}{2}\left(\phi(X,Y)+\phi(JX,JY)\right) $$ and its $J$-anti-invariant part by 
$$\phi^{J,-}(X,Y)=\frac{1}{2}\left(\phi(X,Y)-\phi(JX,JY)\right).$$ We say that $\phi$ is $J$-invariant (resp. $J$-anti-invariant) if $\phi^{J,-}=0$ (resp. $\phi^{J,+}=0$). Moreover, $\phi^{sym}$ (resp. $\phi^{skew}$) is the symmetric
part of $\phi$ with respect to $g$ (resp. anti-symmetric part). We denote by $\alpha^\sharp$ the $g$-Riemannian dual of a $1$-form $\alpha,$ and by $\delta^g$ the codifferential defined as the adjoint of the exterior derivative $d$ with respect to the metric $g.$ 
%We also recall that if $\phi$ is a $J$-invariant $2$-form, then 
%\begin{equation}\label{J-invarinat-formula}
%\phi\wedge\psi\wedge\omega^{n-2}=\frac{1}{n(n-1)}\left( g(\phi,\omega)g(\psi,\omega) - g(\phi,\psi) \right)\omega^n,
%\end{equation}
%for any $2$-form $\psi$, where $(g,\omega)$ is any almost Hermitian structure (see for instance~\cite[Equation 49]{MR742896}).

\subsection{Locally conformally almost K\"ahler metrics.}

We collect here some facts about locally conformally almost K\"ahler metrics. Let $(M, F, \theta)$ be a locally conformally symplectic (LCS for short) manifold of real dimension $2n$, meaning $F$ is a non-degenerate 2-form and $\theta$ is a closed 1-form so that
$$ dF = \theta \wedge F.$$
The 1-form $\theta$ is called {\it the Lee form} of the LCS structure. It is well-known that an exact form $\theta$ corresponds to a globally conformally symplectic structure. Now, suppose that $F$ is compatible with an almost complex structure $J$, that is $g:=F(\cdot,J\cdot)$ is a Riemannian metric. The induced metric $g$ is called a {\it locally conformally almost K\"ahler metric} (LCAK for short). We have the following observations when $T=\theta^\sharp$ is $g$-orthogonal to the image of the Nijenhuis tensor.

 \begin{prop}\label{J-parallel-T}\cite[Proposition 2.1]{MR4698640}
Let $(M, g, J, F, \theta)$ be an LCAK manifold. Suppose that $T=\theta^\sharp$ is $g$-orthogonal to the image of the Nijenhuis tensor. Then, we have $D_TJ=D_{JT}J=0.$ Moreover, $N(T)$ is $g$-symmetric.
 \end{prop}
 \begin{proof}
 For any almost Hermitian structure, the Koszul formula yields for vector fields $X,Y,Z$
 \begin{equation*}
 4g(N(X,Y),JZ)=dF(JX,JY,Z)-dF(X,Y,Z)+2g((D_ZJ)X,Y).
 \end{equation*}
 So $$\left(\iota_Z(dF)\right)^{J,-}=D_ZJ-2N_{JZ}.$$
 For an LCS structure, we have that $\iota_T(dF)$ is $J$-invariant (resp. $\iota_{JT}(dF)=0$ ) hence, $D_TJ=2N_{JT}=0$ (resp. $D_{JT}J=-2N_{T}=0$).
 Moreover, because $(F,\theta)$ is an LCS structure, it follows from~\cite[Proposition 1]{MR1456265}, that the Nijenhuis tensor satisfies
 $$g(N(X,Y),Z)+g(N(Y,Z),X)+g(N(Z,X),Y)=0,$$
 and so $N(T)$ is $g$-symmetric.
 \end{proof}

\begin{prop}\label{j-invariance}\cite[Proposition 2.3]{MR4698640}
Let $(M, g, J, F, \theta)$ be an LCAK manifold. Then
\begin{equation}\label{dJtheta-formula}
 dJ\theta = -\|\theta\|^2 F + \theta \wedge J\theta+2(D \theta)^{J,+}_{J\cdot,\cdot}  + 2N_{JT}.
 \end{equation}
Consequently, the vector field $T=\theta^\sharp$ is $g$-orthogonal to the image of the Nijenhuis tensor if and only if $dJ\theta$ is $J$-invariant.
\end{prop}
 \begin{proof} For the sake of completeness, we give a proof here, slightly different than \cite[Proposition 2.3]{MR4698640}, where a direct computation was used (for an arbitrary almost Hermitian manifold). Applying the Lie derivative $\mathcal{L}_T$ to both sides of $F=g(J\cdot,\cdot)$ and using the Cartan formula, we have that 
 $$ dJ\theta = -\|\theta\|^2F+\theta\wedge J\theta+2\left( D\theta\right)^{sym}_{J\cdot,\cdot}+g(\mathcal{L}_TJ\cdot,\cdot) $$
 %\begin{eqnarray} \label{dJtheta-formula}
 %dJ\theta&=&-\|\theta\|^2F+\theta\wedge J\theta+2\left( D\theta\right)^{sym}_{J\cdot,\cdot}+g(\mathcal{L}_TJ\cdot,\cdot) =\\
 %&=&- \|\theta\|^2 F + \theta \wedge J\theta+2(D \theta)^{J,+}_{J\cdot,\cdot}  + 2N_{JT}.\nonumber
 %\end{eqnarray}
 So the only $J$-anti-invariant part of $dJ\theta$ is given by $\left(\mathcal{L}_TJ\right)^{skew}$. On the other hand, because $$0=d\theta=-2\left(DJ\theta\right)^{sym}_{J\cdot,\cdot}-2g(\mathcal{L}_{JT}J\cdot,\cdot),$$
 we have that $\mathcal{L}_{JT}J$ is $g$-symmetric.
 Now, recall that $\mathcal{L}_{JT}J-J\mathcal{L}_TJ=4N(T,\cdot).$ We deduce that 
 $-J\left(\mathcal{L}_TJ\right)^{skew}=4\left(N(T,\cdot)\right)^{skew}.$
Now it follows from~\cite[Proposition 1]{MR1456265} that for any vector fields $X,Y,Z$, the Nijenhuis tensor satisfies
 $$g(N(X,Y),Z)+g(N(Y,Z),X)+g(N(Z,X),Y)=0.$$
  Hence,
$ \left(N(T,\cdot)\right)^{skew}=-\frac{1}{2}N_T,$ and thus, $$- J\left(\mathcal{L}_TJ\right) = -J\left(\mathcal{L}_TJ\right)^{skew}=-2N_T \; , \mbox{ or } \left(\mathcal{L}_TJ\right) = \left(\mathcal{L}_TJ\right)^{skew} = 2N_{JT}.$$
In particular, if $T$ is $g$-orthogonal to the image of the Nijenhuis tensor, then $\left(\mathcal{L}_TJ\right)^{skew}=0,$ and so $dJ\theta$ is $J$-invariant.
 \end{proof}

\noindent The following proposition contains a computation of the Hodge Laplacian $\Delta^g=d \delta^g + \delta^g d$ of the fundamental $2$-form $F$ of an LCAK structure.

\begin{prop} \label{coexactform}
Let $(M, g, J, F, \theta)$ be a manifold of real dimension $2n$ with an LCAK structure.
Then
%\begin{equation} \label{dJtheta}
%d J\theta = 2(D \theta)^{J,+}_{J\cdot,\cdot} + \theta \wedge J\theta - \|\theta\|^2 F + 2N_{JT}.
%\end{equation}
\begin{equation} \label{ddelF}
d \delta^g F = -(n-1) d J \theta.
\end{equation}
\begin{equation} \label{deldF}
\delta^g d  F =  (\delta^g \theta ) F - 2 N_{JT} + 2(D \theta)^{J,+}_{J\cdot,\cdot} + (n-1)\, \theta \wedge J\theta. 
\end{equation}
\begin{equation} \label{DeltaF}
\Delta^g F = (n-1) \left(\frac{2}{n} \delta^g \theta + \| \theta \|^2\right) F - 2n N_{JT}  - 2(n-2) \big( (D \theta)_0 \big)^{J,+}_{J\cdot,\cdot} \; ,
\end{equation}
where $(\cdot)_{0}$ denotes the trace-free part.  

\noindent We also have
\begin{equation} \label{deldJtheta}
\delta^g d J\theta = (n-2) \big( [T, JT]^{\flat} \big)+ J d(\delta^g \theta + \|\theta\|^2) + \big( \delta^g \theta  + (n-1) \|\theta\|^2 \big)J\theta + 4\delta^g (N_{JT})\; .    
\end{equation}
\end{prop}
\begin{proof} 
Formula (\ref{ddelF}) follows from
$\delta^g F = -(n-1) J \theta $. We check (\ref{deldF}) by a direct computation in coordinates:
\begin{eqnarray*}
 (\delta^g d  F)_{ab} &=& \delta^g (\theta \wedge F)_{ab}\\
 &=& - D_i (\theta \wedge F)_{iab}\\ 
 &=& -D_i ( \theta_i F_{ab} + \theta_a F_{bi} + \theta_b F_{ia})\\
 &=& (\delta^g \theta) F_{ab} - (D_TF)_{ab} + F_{ai} (D_i \theta_b) - F_{bi} (D_i \theta_a) - \theta_a (\delta^g F)_b + \theta_b (\delta^g F)_a \\
 &=& (\delta^g \theta) F_{ab} - 2(N_{JT})_{ab} + 2 ((D \theta)^{J,+}_{J\cdot,\cdot})_{ab} + (n-1) (\theta \wedge J \theta)_{ab},
 \end{eqnarray*}
where for the last equality sign we used $D_T F = 2N_{JT}$, $\delta^g F = -(n-1) J \theta $ and
$$ F_{ai} (D_i \theta_b) - F_{bi} (D_i \theta_a) = 2 ((D \theta)^{J,+}_{J\cdot,\cdot})_{ab}. $$
The equality above can be checked directly from the definition of the $J$-invariant part
$(D \theta)^{J,+}$ of the symmetric tensor $D \theta$. Summing (\ref{ddelF}), (\ref{deldF}) and also using (\ref{dJtheta-formula}), we obtain
$$ (d \delta^g + \delta^g d)  F = \big( \delta^g \theta + (n-1)\|\theta\|^2 \big) F - 2n N_{JT} - 2(n-2) (D \theta)^{J,+}_{J\cdot,\cdot} \; .$$
We decompose the last term on the right-side in its trace (with respect to $F$) and its trace-free part
$$ (D \theta)^{J,+}_{J\cdot,\cdot} = -\frac{(\delta^g \theta)}{2n} F + ((D \theta)^{J,+}_0)_{J\cdot,\cdot} \; , $$
and with this relation (\ref{DeltaF}) follows.

%Formula (\ref{DeltaF}) follows directly from (\ref{dJtheta-formula}), (\ref{ddelF}), (\ref{deldF}). 

Finally, for formula (\ref{deldJtheta}), use (\ref{dJtheta-formula}) in formula (\ref{deldF}), to rewrite the latter as:
$$ \delta^g d  F = dJ \theta +(n-2) \theta \wedge J\theta + ( \delta^g \theta + \| \theta \|^2) F - 4N_{JT}.$$
We then apply one more time divergence on both sides of the above and also use
$$\delta^g (\theta \wedge J \theta) = (\delta^g \theta) J\theta - D_T (J\theta) + D_{JT} \theta  = (\delta^g \theta) J\theta - [T, JT]^{\flat},$$
$$ \delta^g \big[ (\delta^g \theta +\|\theta \|^2) F \big] = - J d\big( \delta^g \theta + \|\theta \|^2 \big) + (\delta^g \theta + \|\theta\|^2) \, \delta^g F,$$
equalities which could be checked by computations in coordinates. For example, for the first one
$$\delta^g (\theta \wedge J \theta)_j = - D_s \big[ (\theta \wedge J\theta)_{sj}\big] = - D_s \big[ \theta_s (J\theta)_j - \theta_j (J\theta)_s \big] =$$
$$ = -(D_s \theta_s) (J\theta)_j - \theta_s (D_s J\theta)_j + (D_s \theta)_j (J\theta)_s + \theta_j (D_s J\theta)_s ,$$
and the last term vanishes as $J\theta $ is coexact. 
\end{proof}

\noindent The following immediate corollary gives a characterization of the condition that $F$ is an eigenform of the Hodge Laplacian, that is, $\Delta^g F = \lambda F$, with $\lambda$ a constant. With this characterization, the eigenform condition for $F$ can be seen as a weak form of the pluricanonical condition introduced in the next subsection (see also Theorem \ref{compchars-plurican}).

\begin{cor} \label{F-eigenv-char}
Let $(M, g, J, F, \theta)$ be a manifold of real dimension $2n$ with an LCAK structure. 

\noindent If $2n \geq 6$, $F$ is an eigenform of the Hodge Laplacian  if and only if 
$$N_T =0, \; \; (D \theta)^{J,+}_0 = 0, \;  \mbox{and } 
\frac{2}{n} \delta^g \theta + \| \theta \|^2 \mbox{ is a constant. } $$
In dimension $2n = 4$, $F$ is an eigenform of the Hodge Laplacian if and only if 
$$N_T =0   , \;  \mbox{and }  \delta^g \theta + \| \theta \|^2 \mbox{ is a constant. } $$  
\end{cor}

%\begin{proof}
%We decompose the last term on the right-side of (\ref{DeltaF}) in its %trace (with respect to $F$) and its trace-free part:
%$$ (D \theta)^{J,+}_{J\cdot,\cdot} = -\frac{(\delta \theta)}{2n} F + ((D \theta)^{J,+}_0)_{J\cdot,\cdot} \; . $$
%With this, relation (\ref{DeltaF}) is rewritten as
%$$ \Delta^g F = (n-1) \big(\frac{1}{n} \delta^g \theta + \| \theta %\|^2\big) F - 2n N_{JT}  - 2(n-2) \big( (D \theta)_0 %\big)^{J,+}_{J\cdot,\cdot} \; .$$
%Note next that if $2n \geq 6$, the condition $ \Delta^g F = \lambda F$ is equivalent with $N_T =0$ and $(D \theta)^{J,+}_0 = 0$, and $\lambda$ if $2n \geq 6$, or just $N_T =0$ if $2n =4$, are equivalent with
%$$ \Delta^g F = \lambda F \; , \mbox{ where } \; \lambda = (n-1) \Big[ \frac{1}{n} \delta^g \theta + \| \theta\|^2 \Big] \; .$$
%    
%\end{proof}

\noindent Next, we prove an integral identity.
\begin{prop} \label{intf1-lcak-prop}
Let $(M, g, J, F, \theta)$ be a compact manifold of real dimension $2n$ with an LCAK structure.
Then
\begin{equation} \label{intf1-lcak}
    0 = \int_M \Big[4 \|(D\theta)^{J,+}_{J\cdot,\cdot}\|^2 - 4\|N_{JT}\|^2 - (\delta^g \theta)^2  + 2(\delta^g \theta) \|\theta\|^2 - n\, J\theta([T,JT]) \Big]
 F^n.
\end{equation} 
Equivalently, the above relation can be rewritten as
\begin{equation} \label{intf1-lcak-v2}
    0 = \int_M \Big[2 \|(D\theta)^{J,+}_{0}\|^2 + n g \left( (D\theta)^{J,+}_{0}  , (\theta \otimes \theta  + J\theta \otimes J\theta ) \right)- 4\|N_{JT}\|^2   - (n-1) (\delta^g \theta) \, \Big(\frac{1}{n} \delta^g \theta + \|\theta \|^2 \Big) \Big]
 F^n.
\end{equation}

\end{prop}
\begin{proof}
    We use the fact that $dJ\theta$ and $\delta^g dF$ are $L^2$-orthogonal and compute using relations (\ref{dJtheta-formula}) and (\ref{deldF}) above.
\begin{eqnarray} \label{orthog-f1}
0 &= &\int_M g(dJ \theta, \delta^g dF) F^n  \\ \nonumber
&= &\int_M \Big[4 \|(D \theta)^{J,+}_{J\cdot,\cdot}\|^2 - 4\|N_{JT}\|^2 +2n\,g\left( (D \theta)^{J,+}_{J\cdot,\cdot}, \theta \wedge J\theta\right)  \\ \nonumber
&+& 2 (\delta^g \theta - \|\theta\|^2)\, g\left( (D \theta)^{J,+}_{J\cdot,\cdot}, F\right) + g \left( \theta \wedge J\theta - \|\theta\|^2 F  ,  (\delta^g \theta) F + (n-1) \theta \wedge J\theta \right)  \Big] F^n.
\end{eqnarray}   
A computation of some of the terms yields the following:
\begin{eqnarray*}
2 g\left( (D \theta)^{J,+}_{J\cdot,\cdot}, \theta \wedge J\theta\right) &=& g \left( (D \theta)^{J,+}, (\theta \otimes \theta + J\theta \otimes J\theta ) \right) \\ &=& g\left(D_T T, T\right)+ g\left(D_{JT}T,JT\right)\\
&=& T(\|\theta\|^2) - J\theta([T,JT])  \\
&=&- \delta^g(\|\theta\|^2 \theta) + (\delta^g \theta)\|\theta \|^2 - J\theta([T,JT]) ;  \\
2g\left( (D \theta)^{J,+}_{J\cdot,\cdot}, F\right) &=& g\left(D\theta, g\right) = - \delta^g \theta  ;\\
g\left(\theta \wedge J\theta - \|\theta\|^2 F  ,  (\delta^g \theta) F + (n-1)\, \theta \wedge J\theta \right) &=& - (n-1) (\delta^g \theta) \|\theta\|^2.
\end{eqnarray*}
These relations replaced in the integral above yield the relation (\ref{intf1-lcak}). For the second version of this integral formula, 
use the decomposition 
$$ (D \theta)^{J,+} = -\frac{(\delta^g \theta)}{2n} g + (D \theta)^{J,+}_0. $$
Thus,
$$ 4 \|(D\theta)^{J,+}_{J\cdot,\cdot}\|^2 = 2\|(D\theta)^{J,+}\|^2 = 2 \|(D\theta)_0^{J,+}\|^2 + \frac{ (\delta^g \theta)^2}{n},$$
where the change of coefficient after the first equal sign happens as
$(D\theta)^{J,+}$ is thought of as a 2-tensor, whereas $(D\theta)^{J,+}_{J\cdot,\cdot}$ as a 2-form, and there is a difference of a factor of 2 in their squared norms.
Also, we have 
$$   2 g\left( (D \theta)^{J,+}_{J\cdot,\cdot}, \theta \wedge J\theta\right)
= -\frac{(\delta^g \theta)}{n}\|\theta\|^2 +  g\left( (D \theta)^{J,+}_0, \theta \otimes \theta + J\theta \otimes J\theta\right). $$
We use these in the initial formula (\ref{orthog-f1}) and after some computations, we obtain the alternative expression of the integral identity 
(\ref{intf1-lcak-v2}).
\end{proof}
\noindent We will also need the following Bochner-type formulas used in the proof of main theorems (see also~\cite{MR4614394} in dimension $4$).
\begin{lem}
Let $(M,g,J,\omega)$ be an almost Hermitian manifold of real dimension $2n$, then for any $1$-form $\alpha$ and vector field $X$
\begin{equation}\label{bochner-formula}
 \left(\delta^g(D\alpha)^{J,+}-\delta^g(D\alpha)^{J,-}\right)(X)=\rho^\ast(\alpha^\sharp,JX)-(n-1)D\alpha(JT,JX)-\sum_{i=1}^{2n}D\alpha(Je_i,\left(D_{e_i}J\right)X),
 \end{equation}
 where $\delta^g$ be the adjoint of the Levi-Civita connection $D$ with respect to the metric $g,$ the vector field $T$ is the $g$-Riemannian dual of the $1$-form $(n-1)\,\theta=J\delta^g\omega$, and $\{e_i,Je_i\}$ is a local $U(n)$ frame.
\end{lem}
\begin{proof}
 \begin{align*}
            \left(\delta^g(D\alpha)^{J,+}-\delta^g(D\alpha)^{J,-}\right)(X)&=-\sum_{i=1}^{2n}\left(D_{e_i}\left((D\alpha)^{J,+}-(D\alpha)^{J,-}\right)\right)(e_i,X)\nonumber\\
            & \begin{multlined}
                =\sum_{i=1}^{2n} -D_{e_i}\left(D_{Je_i}\alpha(JX)\right) +D\alpha(JD_{e_i}e_i,JX)\\+D\alpha(Je_i,JD_{e_i}X)
            \end{multlined} \nonumber\\
            & \begin{multlined}
                =\sum_{i=1}^{2n}-D_{e_i}\left(D_{Je_i}\alpha(JX)\right) + D\alpha(D_{e_i}(Je_i),JX) \\ + D\alpha(Je_i,D_{e_i}(JX)) - D\alpha((D_{e_i}J)e_i,JX) \\ - D\alpha(Je_i,\left(D_{e_i}J\right)X)
                \end{multlined}\nonumber\\
            & \begin{multlined}
                =\sum_{i=1}^{2n}-\left(D_{e_i}\left(D_{Je_i}\alpha\right)\right)(JX) + D\alpha(D_{e_i}(Je_i),JX) \\ - D\alpha((D_{e_i}J)e_i,JX) - D\alpha(Je_i,\left(D_{e_i}J\right)X)
            \end{multlined}\nonumber\\
            &=\sum_{i=1}^{2n}\frac{1}{2}g(R_{e_i,Je_i}\alpha^\sharp,JX)-D\alpha(Je_i,\left(D_{e_i}J\right)X)-(n-1)D\alpha(JT,JX),\nonumber\\
            &=\rho^\ast(\alpha^\sharp,JX)-(n-1)D\alpha(JT,JX)-\sum_{i=1}^{2n}D\alpha(Je_i,\left(D_{e_i}J\right)X).
        \end{align*}
\end{proof}

\begin{cor}
Let $(M, g, J, F, \theta)$ be a compact LCAK manifold. Then,
\begin{eqnarray}\label{LCS-bochner-formula}
 \int_M\|(D\theta)^{J,+}\|^2-\|(D\theta)^{J,-}\|^2\,F^n&=&\int_M\rho^\ast(T,JT)-\left(n-\frac{1}{2}\right)g(D_{JT}\theta,JT)\,F^n\nonumber\\
 %g\left(D\theta,\left(\theta\otimes \theta\right)^{J,+}\right)\nonumber\\
 &-&\int_M\frac{3}{4}\|\theta\|^2\delta^g\theta+2\int_Mg(D\theta,N(T))\,F^n;
 \end{eqnarray}
 \begin{eqnarray}\label{LCS-bochner-formula-JT}
 \int_M \|(DJ\theta)^{J,+}\|^2-\|(DJ\theta)^{J,-}\|^2\,F^n&=&\int_M\rho^\ast(T,JT)-\left(n-\frac{3}{2}\right)g(D_{JT}\theta,JT)+\frac{3}{4}\|\theta\|^2\delta^g\theta\,F^n\nonumber\\
 &+&\int_M\frac{n-1}{2}\|\theta\|^4-2g(\left(DJ\theta\right)^{sym},N(JT))\,F^n\nonumber\\
 &+&2\int_Mg(N_{JT},N(JT))\,F^n.
 \end{eqnarray}
\end{cor}
\begin{proof}
Since $(F,\theta)$ is an LCS structure, we have that
\begin{equation}\label{DJ-expression}
D_XF=\frac{1}{2}\left(X^\flat\wedge J\theta+ JX^\flat\wedge \theta\right)+2N_{JX},
\end{equation}
 where $X^\flat$ is the $g$-Riemannian dual of $X.$ Next, we apply Equation (\ref{bochner-formula}) with  $\alpha=\theta,X=T$.
 We use Equation~\ref{DJ-expression} to compute the last term
\begin{eqnarray*}
& &\sum_{i=1}^{2n} D\theta(Je_i,\left(D_{e_i}J\right)T) =\sum_{i=1}^{2n}g(D_{Je_i}\theta,\left(D_{e_i}J\right)\theta) = \\
&=&\sum_{i,k=1}^{2n}g\left(D_{Je_i}\theta,e_k\right)g\left(\left(D_{e_i}J\right)\theta,e_k\right)= \sum_{i,k=1}^{2n}-g\left(D_{e_i}\theta,e_k\right)g\left(\left(D_{Je_i}J\right)\theta,e_k\right) =\\
&=&\sum_{i,k=1}^{2n}-\frac{1}{2}g\left(D_{e_i}\theta,e_k\right)g\left(\theta(Je_i)J\theta-\theta(e_i)\theta+\|\theta\|^2e_i,e_k\right)+2g\left(D_{e_i}\theta,e_k\right)g(N(T,e_k),e_i) =\\
&=&\sum_{i,k=1}^{2n}\frac{1}{2}g\left(D_{e_i}\theta,e_k\right)g\left(J\theta(e_i)J\theta+\theta(e_i)\theta-\|\theta\|^2e_i,e_k\right)+2g\left(D_{e_k}\theta,e_i\right)g(N(T,e_k),e_i) =\\
&=&\frac{1}{2}g(D_T\theta,\theta)+\frac{1}{2}g(D_{JT}\theta,J\theta)+\frac{1}{2}\|\theta\|^2\delta^g\theta+2g(D\theta,N(T))=\\
&=&\frac{1}{4}d\left(\|\theta\|^2\right)(T)+\frac{1}{2}g(D_{JT}\theta,J\theta)+\frac{1}{2}\|\theta\|^2\delta^g\theta+2g(D\theta,N(T)).
\end{eqnarray*}
Integrating Equation (\ref{bochner-formula}) with $\alpha=\theta,X=T$, we obtain Equation~\ref{LCS-bochner-formula}.
Equation~(\ref{LCS-bochner-formula-JT}) is obtained in a similar way.

\end{proof}
\subsection{Locally conformal structures of the first kind and the pluricanonical condition}

We briefly review the definitions and some basic results due to Vaisman~\cite{MR418003}, but also see the important work of Bazzoni and Marrero \cite{MR3767426}. Let $(M, F, \theta)$ be an LCS manifold of real dimension $2n \geq 4$.
%locally conformally symplectic (LCS for short) manifold of real dimension $2n \geq 4$, meaning $F$ is a non-degenerate 2-form and $\theta$ is a %closed 1-form so that
%$$ dF = \theta \wedge F.$$
%The 1-form $\theta$ is called {\it the Lee form} of the LCS structure. It is well-known that an exact form $\theta$ corresponds to a globally %conformally symplectic structure. 
Let $V$ be the characteristic vector field of the structure, that is, the unique vector field $V$ such that 
$ \iota_V F = \theta$ (where $\iota$ is the contraction). By its definition, $V$ satisfies $\theta(V) = 0$.

In general, an infinitesimal automorphism of the LCS structure is a vector field $X$ such that $\mathcal{L}_X F = 0$. This automatically implies $\mathcal{L}_X \theta = 0$. It is immediate to check that the characteristic vector field $V$ is an infinitesimal automorphism of the LCS structure. The space of infinitesimal automorphisms of the LCS structure, $\mathfrak{X}_{(F,\theta)}$, forms a Lie sub-algebra of $\mathfrak{X}(M)$, the space of all vector fields of $M$, and $\theta :\mathfrak{X}_{(F,\theta)} \rightarrow \mathbb{R}$ induces a Lie algebra morphism, named the {\it Lee morphism} (considering $\mathbb{R}$ as an abelian Lie algebra).

The LCS structure $(F, \theta)$ is said to be {\it of the first kind} if there exists an infinitesimal automorphism $T \in \mathfrak{X}_{(F,\theta)} $ such that $\theta(T) = 1$. The LCS structure $(F, \theta)$ is said to be {\it of the second kind} if for any infinitesimal automorphism $X \in \mathfrak{X}_{(F,\theta)} $, we have $\theta(X) = 0$. In other words, the LCS structure is of the first or the second kind, depending upon whether the Lee morphism is onto or identically zero~\cite{MR809073} (as $\mathbb{R}$ is one-dimensional, these are the only options).

Assume next that $(F, \theta)$ is an LCS structure of the first kind, and let $T$ be an infinitesimal automorphism such that $\theta(T) = 1$. By the definition of $\theta$, this is equivalent to $F(V,T) = 1$. As the characteristic vector field $V$ is determined by the LCS structure $(F, \theta)$ and $T$ is an infinitesimal automorphism of the structure, we automatically have $L_T F = 0$, $L_T \theta = 0$, and also  $[T,V] = 0$. Define the 1-form $\eta$ by $\eta = -\iota_T F$. Since $\theta(T) = 1$, the condition $\mathcal{L}_T F = 0$ is equivalent to
\begin{equation} \label{F1stkind} 
F = d\eta - \theta \wedge \eta.
\end{equation}
Note that $T$ and $V$ satisfy
$$\theta(T) = 1,\quad \theta(V) = 0,\quad \eta(T) =0,\quad \eta(V) =1,$$
$$ \mathcal{L}_T \eta = \mathcal{L}_T \theta = \mathcal{L}_V \eta = \mathcal{L}_V\theta = 0, \quad \iota_T d \eta = \iota_V d\eta = 0.$$
As $F$ is non-degenerate, equation (\ref{F1stkind}) implies that $d\eta$ has rank $2n-2$ and 
$-\theta \wedge \eta \wedge (d\eta)^{n-1}$ is a volume form on $M$, equal to $F^n$, up to a positive constant. In particular, the 1-forms $\theta$ and $\eta$ are nowhere vanishing on $M$, and the same is true for the vector fields $T$ and $V$. Denote by $H$ the regular distribution $\text{Ker}(\theta) \cap \text{Ker}(\eta)$, and note the tangent bundle $TM$ splits smoothly as
$$TM = H \oplus \text{Span}(T, V).$$

\begin{defn}\label{definition-adapted} An almost complex structure $J$ is said to be {\it adapted} to an LCS structure of the first kind if it preserves the above splitting, $J\theta = - \eta$ (this is equivalent to $JV = T$), and
$d \eta(\cdot, J\cdot)$ is a positive definite metric on $H$. Equivalently,
$$ g(X, Y) := F(X, JY) = d\eta(X, JY) - (\theta \wedge \eta)(X, JY) \; $$
is a Riemannian metric on $M$ making the splitting $TM = H \oplus \text{Span}(T,V)$ orthogonal and the vectors $T, V$ orthonormal. 
\end{defn}

\begin{rem}\label{infinite-pluricanonical}
There is an infinite dimensional space of adapted almost complex structures to an LCS structure of the first kind. Indeed, $d\eta$ is a non-degenerate $2$-form on $H$ and so the space of adapted almost complex structures is equivalent to the space of almost complex structures on $H$ compatible with $d\eta$ viewed as a symplectic form.
\end{rem}

We have then
\begin{prop}\label{first-kind-characterization}
If the almost complex structure $J$ is adapted to the LCS structure of the first kind $(F, \theta,\eta)$ inducing a Riemannian metric $g$, then the image of the Nijenhuis tensor is $g$-orthogonal to Span$(T,JT),$ and $D \theta$ is $J$-anti-invariant, where $D$ is the Levi-Civita connection of $g.$
%i.e. $(M^{2n}, g, J)$ is a pluri-canonical almost complex structure.
\end{prop}

\begin{proof} If $g$ is the induced metric, then we have
$J\theta = -\eta$ and $\|\theta\|^2=\|\eta\|^2 = 1$. Relation (\ref{F1stkind}) becomes
$$ F = -d(J\theta) + \theta \wedge J\theta.$$
Comparing this with the expression for $dJ\theta$ from relation (\ref{dJtheta-formula}), we obtain the claimed statement. 
%On the other hand, on any almost Hermitian manifold $(M, g,J, F)$ with Lee 1-form $\theta$, we have (see for example~\cite{MR4698640})
%\begin{equation}\label{formula-almost-hermitian}
% dJ\theta = 2\left(D \theta\right)^{sym,J,+}_{J\cdot,\cdot} + J(d\theta)^{J,-} + 2 N_{JT} + \theta \wedge J\theta -\|\theta\|^2 F ,
% \end{equation}
%where $T=\theta^\sharp$ the $g$-Riemannian dual of $\theta.$ Using $d\theta = 0$ and $ \|\theta\|^2 = 1$, the last two relations imply the statement of the proposition.
\end{proof}

It is natural then to introduce the following definition.
\begin{defn}\label{definition-pluricanonical}
  Let $(g, J, F, \theta)$ be an LCAK structure.
  %Let $(F,\theta)$ be an LCS structure on an almost complex manifold $(M,J)$ and suppose that $F$ is compatible with the almost complex structure $J$ %i.e. $g:=F(\cdot,J\cdot)$ is a Riemannian metric called a {\it locally conformally almost K\"ahler metric} (LCAK for short). 
  Then, $g$ is a {\it pluricanonical} LCAK metric if the image of the Nijenhuis tensor is $g$-orthogonal to Span$(T,JT)$, and $D\theta$ is $J$-anti-invariant, where $D$ is the Levi-Civita connection of the metric $g$ (here and throughout the paper $T=\theta^\sharp$, the $g$-Riemannian dual of $\theta$). 
  \end{defn}
Notice that because the metric tensor $g$ is itself $J$-invariant, the trace of $(D\theta)^{J,-}$ is always zero, so a pluricanonical LCAK metric $g$ is {\it Gauduchon}, that is $\delta^g\theta=0.$ The Gauduchon condition is often alternatively rendered \emph{canonical}, as Gauduchon \cite{MR470920} showed there is such a metric unique up to homothety in each conformal class compatible with $J$, thus the pluricanonical notion is a strengthening of this condition.

We recall that when the $2$-form $F$ of an LCS structure is compatible with an integrable almost complex structure (i.e. when the Nijenhuis tensor $N\equiv 0$), the LCS structure is called {\it locally conformally K\"ahler} (LCK for short). In the LCK setting, pluricanonical LCK metrics were introduced by Kokarev in the context of harmonic maps~\cite{MR2520354} (see also~\cite{MR2566566}). On a compact complex manifold $(M,J)$, Ornea and Verbitsky~\cite{MR4771164} and A. Moroianu and S. Moroianu~\cite{MR3674175} proved that pluricanonical LCK metrics are necessarily {\it Vaisman}, i.e. $\theta$ is parallel with respect to the Levi-Civita connection $D$. However, in the compact non-integrable case, the Lee form of a pluricanonical LCAK metric is not necessarily $D$-parallel, see Examples~\ref{Sec3}.

 \begin{prop}\label{lee-symplectomorphism}
Let $(M, g, J, F, \theta)$ be a pluricanonical LCAK manifold. Then, $\mathcal{L}_TF=0.$
 \end{prop}
 \begin{proof}
 %Applying the Lie derivative $\mathcal{L}_T$ to $F=g(J\cdot,\cdot)$ and using the Cartan formula, we have that 
 %$$dJ\theta=-\|\theta\|^2F+\theta\wedge J\theta+2\left(D\theta\right)^{sym}_{J\cdot,\cdot}+g(\mathcal{L}_TJ\cdot,\cdot).$$
 Since $g$ is pluricanonical LCAK, we have that $\left(D\theta\right)^{sym}_{J\cdot,\cdot}$ is $g$-symmetric. On the other hand,
 it follows from Proposition~\ref{j-invariance} that $\mathcal{L}_TJ$ is $g$-symmetric. Hence, $\mathcal{L}_TF=2\left(D\theta\right)^{sym}_{J\cdot,\cdot}+g(\mathcal{L}_TJ\cdot,\cdot)$ is $g$-symmetric. We conclude from Equation~\ref{dJtheta-formula} that $\mathcal{L}_TF=0.$
 \end{proof}
  \begin{lem}\label{constant-length}\cite{MR4771164}
Let $(M, g, J, F, \theta)$ be a pluricanonical LCAK manifold. Then, the Lee form $\theta$ has a constant length.
 \end{lem}
 \begin{proof}
 It follows from Proposition~\ref{lee-symplectomorphism} that $$dJ\theta=-\|\theta\|^2F+\theta\wedge J\theta.$$
 We apply the exterior derivative to obtain 
\begin{equation*}
0=-\left(d\|\theta\|^2\right)\wedge F-\|\theta\|^2\theta\wedge F-\theta\wedge dJ\theta=-\left(d\|\theta\|^2\right)\wedge F-\|\theta\|^2\theta\wedge F+\|\theta\|^2\theta\wedge F=-\left(d\|\theta\|^2\right)\wedge F.
\end{equation*}
Contracting with $F,$ we get that $\|\theta\|$ is constant.
 \end{proof}
 
\begin{thm}\label{equivalence-first_kind-pluri}
  Let $(M, g, J, F, \theta)$ be an LCAK manifold. Then, the following are equivalent:
\begin{enumerate}[label=\arabic*)]
\item $(F,\theta)$ is of the first kind and $J$ is adapted to $(F,\theta).$\\
\item The metric $g$ is a pluricanonical LCAK metric with a non-zero Lee form $\theta.$
\end{enumerate}
\end{thm}
\begin{proof}
Since the implication $1)\rightarrow 2)$ follows from Proposition~\ref{first-kind-characterization}, we only need to prove $2)\rightarrow 1)$. 
%It follows from Proposition~\ref{lee-symplectomorphism} that $$dJ\theta=-\|\theta\|^2F+\theta\wedge J\theta.$$ 
We have from Lemma~\ref{constant-length} that $\|\theta\|$ is of constant length. We deduce from Proposition~\ref{lee-symplectomorphism} that $J$ is adapted to the LCS structure with $\eta=-J\theta.$
\end{proof}

\subsection{Properties of pluricanonical LCAK metrics}

First, we note that on a pluricanonical LCAK manifold, the distribution Span$(T,JT)$ is integrable, and its integral leaves are totally geodesic. More precisely, we have that

\begin{prop}
Let $(M, g, J, F, \theta)$ be a pluricanonical LCAK manifold. Then, $$D_T\theta=D_{JT}\theta=D_{T}J\theta=D_{JT}J\theta=[T,JT]=0.$$
\end{prop}
\begin{proof}
For any vector field $X$, using Lemma~\ref{constant-length}, we have
$$D_{T}\theta(X)=D_X\theta(T)=g(D_X\theta,\theta)=\frac{1}{2}d\|\theta\|^2(X)=0,$$
$$D_{JT}\theta(X)=D_{T}\theta(JX)=D_{JX}\theta(T)=g(D_{JX}\theta,\theta)=\frac{1}{2}d\|\theta\|^2(JX)=0.$$
Moreover, using Proposition~\ref{J-parallel-T}, we compute
$$D_{T}J\theta(X)=JD_{T}\theta(X)=-D_{T}\theta(JX)=-D_{JX}\theta(T)=-g(D_{JX}\theta,\theta)=-\frac{1}{2}d\|\theta\|^2(JX)=0,$$
$$D_{JT}J\theta(X)=JD_{JT}\theta(X)=-D_{JT}\theta(JX)=-D_{JX}\theta(JT)=g(D_{X}\theta,\theta)=\frac{1}{2}d\|\theta\|^2(X)=0.$$

\end{proof}

It was shown in~\cite{angella2023notecompatibilityspecialhermitian} that a compact complex non-K\"ahler manifold cannot admit
a Vaisman metric and a balanced metric (meaning a metric whose induced 2-form is coclosed). The following is a generalization of that result to the non-integrable case.

\begin{thm}\label{main-thm-1}
Let $(M,F,\theta)$ be a compact manifold of real dimension $2n$ with an LCS structure of the first kind.
%which is not globally conformally symplectic. 
Let $J$ be an adapted almost complex structure for $(F, \theta)$. Then, there is no non-degenerate $2$-form $\omega$ compatible with $J$ such that $d\omega^{n-1}=0.$ In particular, there is no symplectic form $\omega$ compatible with $J$.
 \end{thm}
 \begin{proof}
 Since $J$ is adapted, the induced metric $g:=F(\cdot,J\cdot)$ is pluricanonical LCAK by Theorem~\ref{equivalence-first_kind-pluri}, and we have that $$dJ\theta=-\|\theta\|^2F+\theta\wedge J\theta,$$
 with $\|\theta\|$ a constant. In particular, $-dJ\theta$ is positive semi-definite. 
 Now, we suppose the existence of a non-degenerate $2$-form $\omega$ compatible with $J$ such that $d\omega^{n-1}=0.$ It follows that 
 \begin{eqnarray*}
 \int_MdJ\theta\wedge\omega^{n-1}&=&\int_M J\theta\wedge d\omega^{n-1}=0.
 \end{eqnarray*}
 This is a contradiction with the fact that $-dJ\theta\geq0$ unless $dJ\theta=0,$ requiring $\theta=0.$ But this cannot be true for an LCS structure of the first kind.
 \end{proof}

 The next result gives alternative characterizations of the pluricanonical condition on a compact manifold with a LCAK structure.

 \begin{thm} \label{compchars-plurican}
 Let $(M, g, J, F, \theta)$ be a compact manifold of dimension $2n$ with an LCAK structure that is not almost K\"ahler (i.e. $\theta$ is not identically zero). The following statements are equivalent:

 (i) The structure $(g, J, F, \theta)$ is pluricanonical; 

 (ii) $g$ is Gauduchon, $T$ is orthogonal to the image of the Nijenhuis tensor, and $J\theta([T,JT]) = 0$;

 (iii) The trace-free part of $D \theta$ is $J$-anti-invariant and 
 $\| \theta \|$ is a constant.

\noindent Moreover, if $2n \geq 6$, the above are also equivalent with

(iv) $F$ is an eigenform of the Hodge Laplacian.
 \end{thm}

 \begin{proof}
     The implication (i) $ \Rightarrow$ (ii) follows from the definition of a pluricanonical LCAK structure and Theorem \ref{equivalence-first_kind-pluri}, which guaranties $[T, JT] = 0$, as $(F,\theta)$ is an LCS structure of the first kind. The implication (i) $ \Rightarrow$ (iii) is also clear from the definition and some of the properties of pluricanonical structures that we established above. Indeed, $(D \theta)^{J,+} = 0$ trivially implies $(D \theta)^{J,+}_0 = 0$ and $\delta^g \theta = 0$. From Proposition \ref{lee-form} it also follows that $\| \theta \|$ is a constant. The implication (ii) $ \Rightarrow$ (i) follows directly from relation (\ref{intf1-lcak}) of Proposition \ref{intf1-lcak-prop}, while the implication (iii) $ \Rightarrow$ (i) follows from relation (\ref{intf1-lcak-v2}) of Proposition \ref{intf1-lcak-prop}. Note that under the assumptions of (iii), the integral identity (\ref{intf1-lcak-v2}) takes the form
     $$ 0 = \int_M \Big[- 4\|N_{JT}\|^2   - \frac{n-1}{n}(\delta^g \theta )^2 \Big] F^n,$$
     so it follows that $\delta^g \theta = 0$ and $N_T = 0$. Combined with $(D \theta)^{J,+}_0 = 0$, this implies that the structure is pluricanonical. 

     The implication (i) $ \Rightarrow$ (iv) holds (in all dimensions) by the definition of the pluricanonical condition and Corollary \ref{F-eigenv-char}. Assume next that $2n \geq 6$ and that $F$ is an eigenform of the Hodge Laplacian. By Corollary \ref{F-eigenv-char}, $F$ being an eigenform of the Hodge Laplacian is equivalent with 
     $$N_T =0, \; \; (D \theta)^{J,+}_0 = 0, \;  \mbox{and } 
       \frac{2}{n} \delta^g \theta + \| \theta \|^2 \mbox{ is a constant. } $$
     With these conditions, the integral formula (\ref{intf1-lcak-v2}) of Proposition \ref{intf1-lcak-prop}  becomes
     $$ 0 = \int_M \Big[- (n-1) (\delta^g \theta) \, \Big(\frac{1}{n} \delta^g \theta + \|\theta \|^2 \Big) \Big] F^n = \int_M -(n-1) (\delta^g \theta) \Big(\frac{2}{n} \delta^g \theta + \|\theta \|^2 - \frac{1}{n} \delta^g \theta \Big)  F^n, $$
     so, finally using that $\frac{2}{n} \delta^g \theta + \| \theta \|^2$ is a constant, we get
     $$0 = \frac{n-1}{n} \int_M (\delta^g \theta)^2 \; F^n.$$
     Thus, we get that the metric must be Gauduchon, i.e. $\delta^g \theta = 0$. Together with 
     $(D \theta)^{J,+}_0 = 0$, this implies $(D \theta)^{J,+}= 0$, and, as $N_T=0$ holds, the metric is pluricanonical.
 \end{proof}

% The next three corollaries give alternative characterizations of the %pluricanonical condition on a compact manifold with a LCAK structure. 
 %We first observe the following.

%\begin{cor}\label{eigenform}
%Let $(M, g, J, F, \theta)$ be a compact manifold of real dimension $2n$ %with a Gauduchon LCAK structure. In dimension $2n \geq 6$, the metric $g$ %is pluricanonical LCAK if and only if $F$ is an eigenform of the Hodge %Laplacian. In dimension $2n=4$, the pluricanonical condition on $g$ %implies the eigenform property of $F.$
%\end{cor}
%\begin{proof}
%This is a direct consequence of Proposition~\ref{coexactform}. 
%\end{proof}

%The following is another characterization of the pluricanonical condition.
%\begin{cor}\label{pluricanonical-condition}
%Let $(M, g, J, F, \theta)$ be a compact manifold of real dimension $2n$ %with a Gauduchon LCAK structure and assume that $T$ is orthogonal to the %image of the Nijenhuis tensor. Then the metric $g$ is pluricanonical LCAK %if and only if $J\theta\left([T,JT]\right) = 0.$
%\end{cor}
%\begin{proof}
%    Follows directly from relation (\ref{intf1-lcak}) of Proposition %\ref{intf1-lcak-prop}.
%\end{proof}

\noindent A similar characterization can be obtained for unimodular Lie algebras.  
\begin{thm}\label{lie-algebra-pluricanonical}
Let $\mathfrak{g}$ be a unimodular Lie algebra of real dimension $2n$ equipped with an LCAK structure $(g, J, F, \theta)$. Then the following are equivalent:

(i) The structure is pluricanonical; 

(ii) $T=\theta^\sharp$ is $g$-orthogonal to the image of the Nijenhuis tensor, and $J\theta\left([T,JT]\right) = 0$. 

\noindent Moreover, if $2n \geq 6$, the above are also equivalent with 

(iii) $F$ is an eigenform of the Hodge Laplacian.
\end{thm}

We let the reader verify the details of the proof. The key point is that, since the Lie algebra is unimodular, the divergence of any $1$-form is identically zero. Thus, essentially any argument based on Bochner formulas from compact case will extend to unimodular algebras. If $J$ is integrable, the equivalence (i) $\Leftrightarrow$ (ii) was also established in ~\cite{MR3397500}. It is also noteworthy that the fundamental $2$-form associated with any LCK structure on a $4$-dimensional unimodular Lie algebra satisfies the Laplacian eigenvalue equation.

Now, we would like to prove a generalization of~\cite{MR4771164,MR3674175} to pluricanonical LCAK metrics. We first need the following observation.
 \begin{prop}\label{star-ricci-contracted}
Let $(M, g, J, F, \theta)$ be a compact manifold of real dimension $2n$ with a pluricanonical LCAK structure. Then $$\int_M \rho^\ast(T,JT)\,F^n=0,$$ where $\rho^\ast=R(F)$ is the star-Ricci form and $T=\theta^\sharp$.
 \end{prop}
 \begin{proof}
We recall that for any almost Hermitian manifold, a 2-form representative of the first Chern class $2\pi c_1(M,J)$ is given by
\begin{equation} \label{Chernform} 
\gamma^0(X, Y) = \rho^*(X, Y) + \Phi(X, Y), 
\end{equation}
where $\gamma^0(X,Y)=\frac{1}{2}\sum\limits_{i=1}^{2n}g(R^{\nabla^0}_{X,Y}e_i,Je_i)$ is the Hermitian-Ricci form of the first canonical Hermitian connection $\nabla^0$ (see e.g. \cite{MR1456265}) defined as
$$\nabla^0_X Y = D_X Y - \frac{1}{2}J(D_X J)(Y),$$ 
(here $R^{\nabla^0}$ is the curvature of $\nabla^0$ and $\{e_i,Je_i\}$ is a local $U(n)$ frame of the tangent bundle $TM$), and the 2-form $\Phi$ is given by
\begin{equation*} \label{Phi} 
\Phi(X,Y) = \frac{1}{4} g( J\left(D_X J\right), \left(D_Y J\right)).
\end{equation*}
%It is shown in~\cite{} that
%\begin{equation}
 %- 4 \Phi - dJ \theta = 2 g( N_{J\cdot}, N_{\cdot})+ 2 N_{JT} + \alpha,
%\end{equation}
%where $\alpha$ denotes the 2-form defined by
%\begin{equation*} \label{alpha-def}
 %  \alpha(X, Y) =  (D_{X} \theta)(JY) - (D_{Y} \theta)(JX).
%\end{equation*}
%Hence, using the hypothesis that $T$ is $g$-orthogonal to the span of the Nijenhuis tensor, we get
%\begin{eqnarray*}
%- 4 \Phi(T,JT) - dJ \theta(T,JT) &=& 2 g( N_{JT}, N_{JT})+ 2 g(N(T,JT),JT) + \alpha(T,JT),\\
%&=&-(D_{T} \theta)(T) - (D_{JT} \theta)(JT),\\
%&=&-(D_{T} \theta)(T) + (D_{T} \theta)(T),\\
%&=&0.
%^\end{eqnarray*}
From Proposition~\ref{J-parallel-T}, we deduce that
\begin{equation}\label{vanishing-phi}
\Phi(T,JT) =g(\Phi,\theta\wedge J\theta)=0.
\end{equation}
Since $g$ is a pluricanonical LCAK metric, we deduce from Equation~\ref{deldF} that
\begin{equation}\label{co-exact-theta}
\delta^gdF=(n-1)\,\theta\wedge J\theta.
\end{equation}
We compute using Equations~(\ref{Chernform}),~(\ref{vanishing-phi}) and~(\ref{co-exact-theta})
\begin{eqnarray*}
\int_M g(\rho^\ast,\theta\wedge J\theta)\,F^n&=&\int_M g(\gamma^0,\theta\wedge J\theta)\,F^n-\int_M g(\Phi,\theta\wedge J\theta)\,F^n\\
&=&\int_M g(\gamma^0,\theta\wedge J\theta)\,F^n\\
&=&\frac{1}{n-1}\int_M g(\gamma^0,\delta^gdF)\,F^n\\
&=&\frac{1}{n-1}\int_M g(d\gamma^0,dF)\,F^n\\
&=&0.
%&=&\int_M g(\gamma^0,dJ\theta+ F)F^n\\
%&=&\int_M g(\gamma^0,dJ\theta)F^n+\int_M g(\gamma^0,F)F^n\\
%&=&\int_M g(\gamma^0,F)g(dJ\theta,F)\,F^n-n(n-1)\int_M\gamma^0\wedge dJ\theta\wedge F^{n-2}+\int_M g(\gamma^0,F)F^n\\
%&=&-(n-1)\int_M g(\gamma^0,F)\,F^n-n(n-1)\int_M\gamma^0\wedge dJ\theta\wedge F^{n-2}+\int_M g(\gamma^0,F)F^n\\
%&=&-n(n-1)\int_M\gamma^0\wedge dJ\theta\wedge F^{n-2}+(2-n)\int_M g(\gamma^0,F)F^n.
\end{eqnarray*}
We point out that when $n=1$, $\theta=0$ and so the statement is trivially satisfied. 
%Here, we use in the fourth line equation~(\ref{J-invarinat-formula}) (recall that $dJ\theta$ is $J$-invariant). Now, because $dJ\theta=-%F+\theta\wedge J\theta,$ we have that 
%\begin{eqnarray*}
%\int_M\gamma^0\wedge dJ\theta \wedge F^{n-2}&=&\int_M \gamma^0\wedge J\theta \wedge dF^{n-2}\\
%&=&(n-2)\int_M \gamma^0\wedge J\theta \wedge \theta\wedge F^{n-2}\\
%&=&-(n-2)\int_M \gamma^0\wedge \theta\wedge J\theta \wedge  F^{n-2}\\
%&=& -(n-2)\int_M \gamma^0\wedge dJ\theta \wedge  F^{n-2}-(n-2)\int_M \gamma^0\wedge  F^{n-1},
%\end{eqnarray*}
%and so $$\int_M\gamma^0\wedge dJ\theta \wedge F^{n-2}=-\frac{n-2}{n-1}\int_M\gamma^0\wedge F^{n-1}=-\frac{n-2}{n(n-1)}\int_Mg(\gamma^0,F) F^{n}.$$
%We conclude that 
% $$\int_M g(\rho^\ast,\theta\wedge J\theta)\,F^n=(n-2)\int_M g(\gamma^0,F)F^n+(2-n)\int_M g(\gamma^0,F)F^n=0.$$
 \end{proof}

Now, we can prove a generalization of~\cite{MR4771164,MR3674175} to pluricanonical LCAK metrics.
\begin{thm}\label{generalization-Ornea-Verbitsky}
Let $(M, g, J, F, \theta)$ be a compact manifold of real dimension $2n$ with a pluricanonical LCAK structure. Suppose that 
\begin{equation}\label{theta-killing-condition}
\int_Mg(N(T),D\theta)\,F^n=0,
\end{equation}
where $T=\theta^\sharp$. Then $\theta$ is $D$-parallel.
\end{thm}
\begin{proof}
 
%Since $(F,\theta)$ is an LCS structure, we have that
%\begin{equation}\label{DJ-expression}
%D_XF=\frac{1}{2}\left(X^\flat\wedge J\theta+ JX^\flat\wedge \theta\right)+2N_{JX},
%\end{equation}
% where $X^\flat$ is the $g$-Riemannian dual of $X.$ Next, we apply (\ref{bochner-formula}) with  $\alpha=\theta,X=T$. Using the fact that $g$ is %pluricanonical LCAK (in particular $N(T)$ is $g$-symmetric by Proposition~\ref{J-parallel-T}), we obtain
%\begin{eqnarray*}
%-\left(\delta^g\left(D\theta\right)\right)(T)&=&\rho^\ast(T,JT)-(n-1)D\theta(JT,JT)-\sum_{i=1}^{2n}D\theta(Je_i,\left(D_{e_i}J\right)\theta)\\
%&=&\rho^\ast(T,JT)+\frac{n-1}{2}g(d\|\theta\|^2,\theta)-\sum_{i=1}^{2n}\frac{1}{2}g(D_{Je_i}\theta,\theta(e_i)J\theta+\theta(Je_i)\theta-%\|\theta\|^2Je_i)\\
%&+&2g(N(T,D_{Je_i}\theta),Je_i)\\
%&=&\rho^\ast(T,JT)-\frac{1}{2}g(D_{JT}\theta,J\theta)-\frac{1}{2}g(D_{T}\theta,\theta)-\delta^g\theta-2\sum_{i=1}^{2n}g(N(T,Je_i),D_{Je_i}\theta))\\
%&=&\rho^\ast(T,JT)+\frac{1}{2}g(D_{T}\theta,\theta)-\frac{1}{4}g(d\|\theta\|^2,\theta)-2\sum_{i=1}^{2n}g(N(T,Je_i),D_{Je_i}\theta))\\
%&=&\rho^\ast(T,JT)+\frac{1}{4}g(d\|\theta\|^2,\theta)-2\sum_{i=1}^{2n}g(N(T,Je_i),D_{Je_i}\theta))\\
%&=&\rho^\ast(T,JT)-2g(N(T),D\theta),
%\end{eqnarray*}
%where we use equation~(\ref{DJ-expression}) in the second line.
Since $g$ is pluricanonical LCAK, we have that
$$g(D_{JT}\theta,JT)=-g(D_{T}\theta,T)=-\frac{1}{2}d(\|\theta\|^2)(T).$$
We integrate then Equation~(\ref{LCS-bochner-formula}) using the above equation and Proposition~\ref{star-ricci-contracted}, we obtain
$$\int_M\|D\theta\|^2\,F^n=2\int_Mg(N(T),D\theta)\,F^n=0.$$
We conclude that $\theta$ is $D$-parallel.
\end{proof}
It follows from the above theorem that for a pluricanonical LCAK metric $g$ on a compact manifold, $T$ is a Killing holomorphic vector field if and only if the condition~(\ref{theta-killing-condition}) is satisfied ($T$ being a (real) holomorphic vector field means that $\mathcal{L}_TJ=0$). On the other hand, in the integrable setting,
it is known that for a Vaisman metric, the vectors fields $T$ and $JT$ must be Killing holomorphic, see for instance~\cite{MR4771164}. It is natural then to ask in the non-integrable case for a pluricanonical LCAK metric $g$, when is $JT$ a Killing holomorphic vector field?
\begin{thm}
Let $(M, g, J, F, \theta)$ be a compact manifold of real dimension $2n$ with a pluricanonical LCAK structure. Suppose that $$\int_Mg(N(JT),D\left(J\theta\right))\,F^n=0,$$ where $T=\theta^\sharp$. Then, $JT$ is a Killing holomorphic vector field.
\end{thm}
\begin{proof}
Applying the Lie derivative with respect to $JT$ to $F=g(J\cdot,\cdot)$, we get
$$2\left(DJ\theta\right)^{sym}_{J\cdot,\cdot}=-g(\mathcal{L}_{JT}J\cdot,\cdot).$$
Hence, $\left(DJ\theta\right)^{sym}$ is $J$-anti-invariant. On the other hand, $\left(DJ\theta\right)^{skew}=\frac{1}{2}dJ\theta$ is $J$-invariant by Lemma~\ref{j-invariance}, and $N(JT)$ is $g$-symmetric by Proposition~\ref{J-parallel-T}. When $g$ is pluricanonical LCAK, Equation~(\ref{LCS-bochner-formula-JT}) then reduces to
\begin{eqnarray}
\int_M\frac{1}{2}\|dJ\theta\|^2-\|(DJ\theta)^{J,-}\|^2\,F^n&=&\int_M\rho^\ast(T,JT)\,F^n+\frac{1}{2}(n-\frac{3}{2})\int_Md\left(\|\theta\|^2\right)(T)\,F^n
 +\frac{n-1}{2}\int_M\|\theta\|^4\,F^n\nonumber\\
 &-&2\int_Mg(\left(DJ\theta\right)^{sym},N(JT))\,F^n\nonumber\\
 &=&\int_M\rho^\ast(T,JT)\,F^n +\frac{n-1}{2}\int_M\|\theta\|^4\,F^n-2\int_Mg(\left(DJ\theta\right)^{sym},N(JT))\,F^n.\label{Lee-vector-killing}
\end{eqnarray}
%\begin{eqnarray}
 %\left(\delta^g(DJ\theta)^{J,+}-\delta^g(DJ\theta)^{J,-}\right)(JT)&=&\rho^\ast(T,JT)-(n-\frac{3}{2})g(D_{T}\theta,T)-%\|\theta\|^2g\left(\left(DJ\theta\right)^{J,+},F\right)\nonumber\\
 %&-&2g(\left(DJ\theta\right)^{sym},N(JT))\nonumber\\
%\frac{1}{2} \left(\delta^gdJ\theta\right)(JT)-\delta^g(DJ\theta)^{sym}(JT)&=& \rho^\ast(T,JT)-\frac{1}{2}(n-\frac{3}{2})d(\|\theta\|^2)(T)-\frac{1}%{2}\|\theta\|^2g\left(dJ\theta,F\right)\nonumber\\
 %&-&2g(\left(DJ\theta\right)^{sym},N(JT))\nonumber\\
 %&=&\rho^\ast(T,JT)-\frac{1}{2}\|\theta\|^2(-(n-1)\|\theta\|^2-\delta^g\theta)-2g(\left(DJ\theta\right)^{sym},N(JT))\nonumber\\
 %&=&\rho^\ast(T,JT)+\frac{n-1}{2}\|\theta\|^4-2g(\left(DJ\theta\right)^{sym},N(JT))\label{Lee-vector-killing}
 %\end{eqnarray}
On the other hand, since $g$ is pluricanonical LCAK, Equation~(\ref{deldJtheta}) reduces to
\begin{equation*}
\delta^g d J\theta = (n-2)\left(D_{T}J\theta-D_{JT}\theta\right) + (n-1)\|\theta\|^2J\theta.    
\end{equation*}
Hence
\begin{equation*}
\delta^g d J\theta(JT)=(n-2)d(\|\theta\|^2)(T)+(n-1)\|\theta\|^4=(n-1)\|\theta\|^4,
\end{equation*}
because $\|\theta\|$ is constant. Substituting this in Equation~(\ref{Lee-vector-killing}), we obtain
\begin{eqnarray*}
\frac{n-1}{2}\int_M\|\theta\|^4\,F^n-\int_M\|(DJ\theta)^{J,-}\|^2\,F^n&=&\int_M\rho^\ast(T,JT)\,F^n +\frac{n-1}{2}\int_M\|\theta\|^4\,F^n-2\int_Mg(\left(DJ\theta\right)^{sym},N(JT))\,F^n\\
-\int_M\|(DJ\theta)^{J,-}\|^2\,F^n&=&\int_M\rho^\ast(T,JT)\,F^n-2\int_Mg(\left(DJ\theta\right)^{sym},N(JT))\,F^n.
\end{eqnarray*}
The theorem follows from Proposition~\ref{star-ricci-contracted}.

\end{proof}

In~\cite{MR1456265}, Gauduchon introduced a distinguished $1$-parameter family of canonical Hermitian connections $\nabla^t$ that includes the Chern connection (when $t=1$), the first canonical connection $\nabla^0$, and the Bismut connection (when $t=-1$) for almost Hermitian metrics. The Hermitian-Ricci form $\gamma^t$ of $\nabla^t$ is defined as $$\gamma^t(X,Y)=\frac{1}{2}\sum\limits_{i=1}^{2n}g(R^{\nabla^t}_{X,Y}e_i,Je_i),$$
where $R^{\nabla^t}$ is the curvature of $\nabla^t$. In~\cite{barbaro2025calabiyaulocallyconformallykahler}, Barbaro and Otiman proved that, in the compact case, LCK Hermitian-Ricci flat metrics, with respect to any canonical Hermitian connections $\nabla^t$, are Vaisman provided that the metric is Gauduchon. In the following, we generalize their result to the LCAK setting.
\begin{thm}\label{Gauduchon-Ricci-flat}
Let $(M, g, J, F, \theta)$ be a compact manifold of real dimension $2n$ with an LCAK structure with $T=\theta^\sharp$ $g$-orthogonal to the image of the Nijenhuis tensor $N$. Suppose that $g$ is Gauduchon and Hermitian-Ricci flat with respect to some connection $\nabla^t$, and that $$\int_Mg(N(T),D\theta)\,F^n=0,$$ then $\theta$ is $D$-parallel.
%, we have the following
%$\begin{enumerate}[label=\arabic*)]
%\item if $\int_Mg([T,JT],JT)\,F^n=0$, then $g$ is a pluricanonical LCAK metric.\\
%\item  
%\end{enumerate}
\end{thm}
\begin{proof} First, note that under the assumptions, formula (\ref{intf1-lcak}) of Proposition \ref{intf1-lcak-prop} becomes
$$ \int_M  \|(D\theta)^{J,+}\|^2 F^n = \frac{n}{2}\int_M J\theta([T,JT]) \,F^n = - \frac{n}{2}\int_M g(D_{JT}\theta,J\theta) \, F^n \; .$$
Using again the assumptions, the above formula, and relation (\ref{LCS-bochner-formula}), we also get: 
$$ \int_M  \|(D\theta)^{J,-}\|^2 F^n = \int_M \Big[ -\rho^{\ast}(T, JT) + \frac{n-1}{2} g(D_{JT}\theta,J\theta) \Big] \; F^n \; .$$
We compute the first term on the right hand-side. Because $T$ is $g$-orthogonal to the image of the Nijenhuis tensor, we get from Relation~(\ref{Chernform}) that 
$$\rho^\ast(T,JT) = \gamma^0(T,JT) \; .$$
From Gauduchon~\cite[Remark 5]{MR1456265}, we deduce the following relation between Hermitian-Ricci forms of the $\nabla^t$ connections in the almost Hermitian case:
$$\gamma^t=\gamma^0-\frac{t(n-1)}{2}dJ\theta,$$
for $-1\leq t\leq 1.$ By hypothesis, the metric is Hermitian--Ricci flat with respect to $\nabla^t$ for some $t$, so
\begin{equation*}\label{hermitian-flat-relation}
\gamma^0=\frac{t(n-1)}{2}dJ\theta.
\end{equation*}
Using this, and again the other assumptions, we have
\begin{eqnarray*}
    \int_M \rho^\ast(T,JT) \; F^n &=& \int_M \gamma^0(T,JT) \; F^n = \int_M g\left( \gamma^0, \theta \wedge J\theta \right) \; F^n = \frac{t(n-1)}{2}\int_M g\left( dJ\theta , \theta \wedge J\theta \right) \; F^n \\
    &=& \frac{t(n-1)}{2}\int_M g\left( J\theta , \delta^g(\theta \wedge J\theta ) \right) \; F^n \\ &= & \frac{t(n-1)}{2}\int_M g\left( J\theta , (\delta^g \theta) J\theta - D_T J \theta +D_{JT} \theta \right) \; F^n \\
    & =& \frac{t(n-1)}{2}\int_M g(D_{JT}\theta,J\theta) \, F^n \; .
\end{eqnarray*}
We replace this in the above integral to conclude
\begin{equation}\label{J-anti-D-Lee}
\int_M\|\left(D\theta\right)^{J,-}\|^2\,F^n=(1-t)\frac{(n-1)}{2}\int_Mg(D_{JT}\theta,J\theta)\,F^n.
\end{equation}
But recall that at the start of this proof we also established
\begin{equation} \label{J-inv-D-Lee}
    \int_M  \|(D\theta)^{J,+}\|^2 F^n = - \frac{n}{2}\int_M g(D_{JT}\theta,J\theta) \, F^n \; .
\end{equation}
Since $-1\leq t\leq 1,$ we have two cases:
\begin{enumerate}
\item if $t<1$, then Equations~(\ref{J-anti-D-Lee})  and (\ref{J-inv-D-Lee}) can be satisfied only when $$\int_M\|\left(D\theta\right)^{J,+}\|^2\,F^n=\int_M\|\left(D\theta\right)^{J,-}\|^2\,F^n=\int_Mg(D_{JT}\theta,J\theta)\,F^n=0,$$ and the theorem follows.\\
\item if $t=1$, then we conclude from (\ref{J-anti-D-Lee}) that $\left(D\theta\right)^{J,-}=0$. In this case, the result follows from Theorem~\ref{anti-pluricanonical} below.
\end{enumerate}
\end{proof}

 \section{Anti-Pluricanonical Locally conformally almost-K\"ahler metrics}\label{sec-anti-plu}

 So far, we investigated pluricanonical LCAK metrics $g$ characterized by the fact that $D\theta$ is $J$-anti-invariant. It is natural then to also examine LCAK metrics with the property that $D\theta$ is $J$-invariant. For this purpose, we introduce the following definition.
  
  \begin{defn}\label{def-anti-pluricanonical}
  Let $(F,\theta)$ be an LCS structure on an almost complex manifold $(M,J)$ and suppose that $F$ is compatible with the almost complex structure $J$ inducing an LCAK metric $g$. Then $g$ is called an {\it anti-pluricanonical} LCAK metric if the image of the Nijenhuis tensor is $g$-orthogonal to Span$(T,JT)$ and $D\theta$ is $J$-invariant, where $D$ is the Levi-Civita connection of $g$.
  \end{defn}

 It turns out that the anti-pluricanonical condition is equivalent to the fact that the Lee field $T=\theta^\sharp$ is a real holomorphic vector field (i.e. $\mathcal{L}_TJ=0$).

 \begin{prop}\label{holomorphic}
Let $(M, g, J, F, \theta)$ be an LCAK manifold. Then, $g$ is an anti-pluricanonical LCAK metric if and only if $T=\theta^\sharp$ is a (real) holomorphic vector field.
 \end{prop}
 \begin{proof}
% Applying the Lie derivative $\mathcal{L}_T$ to $F=g(J\cdot,\cdot)$ and using the Cartan formula, we have that 
% \begin{equation}\label{cartan}
% dJ\theta=-\|\theta\|^2F+\theta\wedge J\theta+2\left(D\theta\right)^{sym}_{J\cdot,\cdot}+g(\mathcal{L}_TJ\cdot,\cdot).
% \end{equation}
 
 First suppose that $g$ is an anti-pluricanonical LCAK metric; then $\left(D\theta\right)^{sym}_{J\cdot,\cdot}$ is a $J$-invariant $2$-form, and we have  that $\left(\mathcal{L}_TJ\right)^{sym}=0$. Moreover, it follows from Proposition~\ref{j-invariance} that $\left(\mathcal{L}_TJ\right)^{skew-sym}=0$. Hence, $\mathcal{L}_TJ=0$, and so $T$ is a holomorphic vector field. The converse follows readily from Equation~(\ref{dJtheta-formula}).
 \end{proof}

In the integrable case, it was proved by A. Moroianu, S. Moroianu and Ornea~\cite{MR3830777} that on a compact LCK manifold, if $T=\theta^\sharp$ is a holomorphic vector field, then the LCK metric $g$ is Vaisman whenever $g$ is Gauduchon. We prove a generalization of this result to the non-integrable case.
 \begin{thm}\label{anti-pluricanonical}
 Let $(M, g, J, F, \theta)$ be a compact manifold of real dimension $2n$ with an LCAK structure.
 Suppose that the image of the Nijenhuis tensor is $g$-orthogonal to Span$(T,JT)$.
 Then, $g$ is an anti-pluricanonical LCAK Gauduchon metric if and only if $\theta$ is $D$-parallel.
 \end{thm}
 \begin{proof} Just the ``only if" direction needs proof. This follows directly from the integral identity (\ref{intf1-lcak}) of Proposition \ref{intf1-lcak-prop}. Indeed, if $D \theta$ is assumed $J$-invariant then
 $$ g(D_{JT} T, JT) = (D_{JT} \theta)(JT) = (D_{T} \theta)(T) = \frac{1}{2} T(\| \theta \|^2) \; ,$$
 so the last term on the right-hand side of relation (\ref{intf1-lcak}) vanishes:
 $$ J\theta([T,JT]) = g([T,JT], JT) = g(D_T JT, JT) - g(D_{JT} T, JT) = \frac{1}{2} T(\| \theta \|^2) - \frac{1}{2} T(\| \theta \|^2) = 0 \; .$$
 Using also assumptions $N_T = 0$, and $\delta^g \theta = 0$, the relation (\ref{intf1-lcak}) reduces to 
 $$ \int_M 4 \| (D \theta)^{J,+}_{J\cdot, \cdot} \|^2 F^n = 0 \; .$$
 Thus, using again the anti-pluricanonical assumption, it follows  that
 $ D \theta = (D \theta)^{J,+} = 0 $.
 \end{proof}

 \vspace{0.2cm}

\section{Explicit examples of Pluricanonical LCAK Lie algebras}\label{Sec3}
Plenty of examples of LCS structures of the first kind were presented in~\cite{MR3767426} (see also~\cite{MR3763412}). As a consequence of Remark~\ref{infinite-pluricanonical} and Theorem~\ref{equivalence-first_kind-pluri}, many pluricanonical LCAK examples can then be obtained from~\cite{MR3767426}.
\subsection{Pluricanonical LCAK $4$-dimensional Lie algebras}\label{LCAKdimension 4}

%\subsection{Explicit Examples}\label{Sec3} 
In the following, we present compact examples of pluricanonical LCAK manifolds of dimension $4$. In the first example, the Lee form is $D$-parallel, whereas it is not in the second and third examples. The examples are Lie algebras such that their associated Lie groups admit compact quotients (here we use the same notation of Lie algebras as~\cite{MR404362}).
%The following are examples of Lie algebras that can be equipped with pluricanonical LCAK metrics.
%Moreover,. 
%This actually answers a question of~\cite{MR4698640} about the existence of such examples.
It follows from Theorem~\ref{main-thm-1} that none
of these almost complex Lie groups admit a symplectic form that is compatible with the given almost complex structure. By the classification of $4$-dimensional Lie algebras~\cite{MR2307912}, the first two examples are symplectic but not the third.

\begin{enumerate}
\item{The Lie algebra $\mathcal{A}_{3,6}\oplus \mathcal{A}_1$}: the structure of the Lie algebra is 
$$[e_1,e_3]=-e_2,\quad [e_2,e_3]=e_1,$$
where $\{e_1,e_2,e_3,e_4\}$ is a basis of $\mathcal{A}_{3,6}\oplus \mathcal{A}_1$.
The associated simply connected group to the Lie algebra $\mathcal{A}_{3,6}\oplus \mathcal{A}_1$ admits lattices 
(see for example~\cite{MR3480018,MR3763412,MR4088745}, in the notation of~\cite{MR4088745} $\mathcal{A}_{3,6}\oplus \mathcal{A}_1$ corresponds to $\mathfrak{r}\mathfrak{r}^\prime_{3,0}$).
We consider the non-integrable almost complex structure
$$Je_1=e_3,\quad Je_2=e_4.$$
%The almost complex structure $J$ is non-integrable because $N(e_1,e_2)=\frac{1}{4}e_3.$ We consider the following metric $g$ compatible with $J$
%$$g=\sum_{i=1}^4e^i\otimes e^i,$$
%where $\{e^1,e^2,e^3,e^4\}$ is the dual basis. The pair $(J,g)$ induces the fundamental form
Then, the fundamental $2$-form
$$F=e^{13}+e^{24}$$
is compatible with $J$ (here $e^{12}=e^1\wedge e^2$ etc, where $\{e^1,e^2,e^3,e^4\}$ is the dual basis). The pair $(J,F)$ induces an LCAK metric with the Lee form given by
$$\theta=e^4.$$
%Hence $d\theta=0.$ 
Moreover, $\theta$ is $D$-parallel. We also remark that the image of the Nijenhuis tensor is Span$(e_1,e_3)$ so $T=\theta^\sharp=e_4$ is $g$-orthogonal to the image of the Nijenhuis tensor, and that the metric is Chern flat (that is the curvature of the Chern connection vanishes).\\

\item{The Lie algebra $\mathcal{A}_{4,1}$}: the structure of the Lie algebra is 
$$[e_2,e_4]=e_1,\quad [e_3,e_4]=e_2,$$
where $\{e_1,e_2,e_3,e_4\}$ is a basis of $\mathcal{A}_{4,1}$.
The associated simply connected group to the Lie algebra $\mathcal{A}_{4,1}$ admits lattices 
(see for example~\cite{MR3480018,MR3763412,MR4088745}, in the notation of~\cite{MR4088745} $\mathcal{A}_{4,1}$ corresponds to $\mathfrak{n}_{4}$). We consider the non-integrable almost complex structure
$$Je_1=e_3,\quad Je_2=e_4.$$
%The almost complex structure $J$ is non-integrable because $N(e_1,e_2)=\frac{1}{4}e_2.$ We consider the metric $g$ compatible with $J$ given by
%$$g=\sum_{i=1}^4e^i\otimes e^i,$$
%where $\{e^1,e^2,e^3,e^4\}$ is the dual basis. The pair $(J,g)$ induces the fundamental form
Then, the fundamental $2$-form
$$F=e^{13}+e^{24}$$
%Notice that the form $dF=e^{234}$. 
is compatible with $J$.  The pair $(J,F)$ induces an LCAK metric $g$ with the Lee form given by
$$\theta=-e^3.$$
In this example, $g$ is a pluricanonical LCAK metric. Indeed $$(D\theta)^{sym}=(D\theta)^{sym,J,-}=\frac{1}{2}\left(e^2\otimes e^4 +e^4\otimes e^2\right).$$
We also remark that the span of the Nijenhuis tensor is Span$(e_2,e_4)$, so $T=-e_3$ is $g$-orthogonal to the image of the Nijenhuis tensor.\\

\item{The Lie algebra $\mathcal{A}_{4,8}$}: the structure of the Lie algebra is 
$$[e_2,e_3]=e_1,\quad [e_2,e_4]=e_2,\quad [e_3,e_4]=-e_3,$$
where $\{e_1,e_2,e_3,e_4\}$ is a basis of $\mathcal{A}_{4,8}$.
The associated simply connected group to the Lie algebra $\mathcal{A}_{4,8}$ admits lattices 
(see for example~\cite{MR3480018,MR4088745}, in the notation of~\cite{MR4088745} $\mathcal{A}_{4,8}$ corresponds to $\mathfrak{d}_{4}$).
We consider the non-integrable almost complex structure
$$Je_1=e_4,\quad Je_2=e_3.$$
%The almost complex structure $J$ is non-integrable because $N(e_1,e_2)=\frac{1}{2}e_3.$ We consider the metric $g$ compatible with $J$ given by
%$$g=\sum_{i=1}^4e^i\otimes e^i,$$
%where $\{e^1,e^2,e^3,e^4\}$ is the dual basis. The pair $(J,g)$ induces the fundamental form
Then, the fundamental $2$-form
$$F=e^{14}+e^{23}$$
is compatible with $J.$
%Remark that here the form $dF=-e^{234},$ and moreover,
The pair $(J,F)$ induces an LCAK metric $g$ with the Lee form given by
$$\theta=-e^4.$$
%Hence $d\theta=0.$ 
The metric  $g$ is a pluricanonical LCAK metric. Indeed $$(D\theta)^{sym}=(D\theta)^{sym,J,-}=e^3\otimes e^3 -e^2\otimes e^2.$$
Furthermore, the image of the Nijenhuis tensor is Span$(e_2,e_3)$ so $T=-e_4$ is $g$-orthogonal to the image of the Nijenhuis tensor.\\

\end{enumerate}

\subsection{Pluricanonical LCAK $6$-dimensional Lie algebras}\label{LCAK-dimension 6}
In the following, we present two examples of non-LCK, LCAK nilpotent Lie algebras such that the Lee form is $D$-parallel. The examples are from the classification of locally conformal symplectic nilpotent Lie algebras in~\cite{MR3767426}. Since the structure equations are rational, the associated Lie groups admit compact quotients. Moreover, both examples are symplectic. Here, we use the notation of~\cite{MR3767426}.

\begin{enumerate}

\item{The Lie algebra $\mathcal{L}_{5,3}\oplus \mathcal{A}_1$}: the structure of the Lie algebra is 
$$[e_1,e_2]=e_5,\quad [e_2,e_3]=e_6,\quad [e_1,e_5]=-e_6,$$
where $\{e_1,e_2,e_3,e_4,e_5,e_6\}$ is a basis of $\mathcal{L}_{5,3}\oplus \mathcal{A}_1$.
We consider the non-integrable almost complex structure
$$Je_1=e_5,\quad Je_2=e_3,\quad Je_4=e_6.$$
%The almost complex structure $J$ is non-integrable because $N(e_1,e_2)=\frac{1}{4}e_5.$ We consider the metric $g$ compatible with $J$ given by
%$$g=\sum_{i=1}^6e^i\otimes e^i,$$
%where $\{e^1,e^2,e^3,e^4,e^5,e^6\}$ is the dual basis. The pair $(J,g)$ induces the fundamental form
The fundamental $2$-form
$$F=e^{15}+e^{23}+e^{46}$$
is compatible with $J$, and induces an LCAK metric $g$ with the Lee form given by
%Remark that here the form $dF=e^{415}+e^{423},$ and moreover, the Lee form is given by
$$\theta=e^4.$$
%Hence $d\theta=0.$ 
Moreover, it is easy to check that $D\theta=0.$ We also remark that $dJ\theta$ is $J$-invariant, hence the image of the Nijenhuis tensor is $g$-orthogonal to $T=e_4$.\\

\item{The Lie algebra $\mathcal{L}_{5,6}\oplus \mathcal{A}_1$}: the structure of the Lie algebra is 
$$[e_1,e_2]=e_4,\quad [e_1,e_4]=e_5,\quad [e_1,e_5]=e_6,\quad [e_2,e_4]=-e_6,$$
where $\{e_1,e_2,e_3,e_4,e_5,e_6\}$ is a basis of $\mathcal{L}_{5,6}\oplus \mathcal{A}_1$.
We consider the non-integrable almost complex structure
$$Je_1=e_5,\quad Je_2=e_4,\quad Je_3=e_6.$$
%The almost complex structure $J$ is non-integrable because $N(e_1,e_2)=\frac{1}{4}(e_1-e_4).$ We consider the metric $g$ compatible with $J$ given by
%$$g=\sum_{i=1}^6e^i\otimes e^i,$$
%where $\{e^1,e^2,e^3,e^4,e^5,e^6\}$ is the dual basis. The pair $(J,g)$ induces the fundamental form
The fundamental $2$-form
$$F=e^{15}+e^{24}+e^{36}$$
is compatible with $J$, and induces an LCAK metric $g$ with the Lee form given by
%Remark that here the form $dF=e^{315}+e^{324},$ and moreover, the Lee form is given by
$$\theta=e^3.$$
%Hence $d\theta=0.$ 
Moreover, we have that $D\theta=0.$ We also remark that $dJ\theta$ is $J$-invariant, hence the image of the Nijenhuis tensor is $g$-orthogonal to $T=e_3$.\\

\end{enumerate}

\subsection{Classification of unimodular almost abelian pluricanonical LCAK $4$-dimensional Lie algebras}

 An {\it almost abelian} Lie group $G$ is a Lie group whose Lie algebra $\g$ has a codimension-one abelian ideal $\n\subset \g$. Given an almost Hermitian left-invariant structure $(g,J)$ on a $2n$-dimensional almost abelian Lie group $G$, define $\n_1:=\n\cap J\n$ and $J_1: = J_{|_{\n_1}}$. Then we can choose an orthonormal basis $\{e_1,\ldots,e_{2n}\}$ for $\g$ such that 

   $$ \n=\mbox{Span}_\R(e_1,\ldots,e_{2n-1}) \quad \text{ and } \quad Je_i=e_{2n-i+1} \text{ for } i=1,\ldots,n.  $$

    Hence, the fundamental form $F(\cdot,\cdot):=g(J\cdot,\cdot)$ associated to the almost Hermitian structure $(J,g)$ is
    $$ F = e^1\wedge e^{2n} + e^2\wedge e^{2n-1} + \cdots + e^n\wedge e^{n+1},$$
    given in terms of the dual left-invariant frame $\{e^1,\ldots,e^{2n}\}$.

    The algebra structure of $\g$ is completely described by the adjoint map 
    \begin{eqnarray*}  
    \mbox{ad}_{e_{2n}}: \g& \rightarrow& \g \\
     x &\mapsto& [e_{2n},x].\end{eqnarray*}
     The matrix associated to this endomorphism is
    
    \begin{equation}\label{eq: ad_e_2n}
        \ad_{e_{2n}|_\n} = \begin{pmatrix} 
        a & b \\
        v & A
        \end{pmatrix}, \quad a \in\R,\; b,v \in \n_1,\; A\in\mathfrak{gl}(\n_1).
    \end{equation}
    
    The data $(a,b,v,A)$ completely characterize the almost Hermitian structure $(J,g)$. For example, the integrability of $J$ can be expressed in terms of $(a,b,v,A)$ asking that $b=0$ and $A\in\mathfrak{gl}(n_1,J_1)$, where $\mathfrak{gl}(n_1,J_1)$ denotes endomorphisms of $\n_1$ commuting with $J_1$, see \cite[Lemma 4.1]{MR3957836}.
    
    On an almost abelian almost Hermitian Lie group $\left(G,[\cdot,\cdot]_{(a,b,v,A)},J,g\right),$ the Lee form is given by
    \begin{equation}\label{lee-form}
        (n-1)\theta = J\delta^gF = (Jv)^{\flat} - (\tr\,A)e^{2n},
    \end{equation}
    with respect to the adapted unitary basis $\{e_1,\ldots,e_{2n}\}$, see for example~\cite{MR4466741}.
    
    Almost abelian Lie algebras admitting LCS structures were studied in~\cite{MR3763412}. In the following, we classify almost abelian unimodular $4$-dimensional Lie algebras that can be equipped with a pluricanonical LCAK metric. We suppose that the dimension is $4$ and that the Lie algebra is unimodular. The condition that $\theta$ is $d$-closed implies that
    \begin{align*}
    A_{21}v_1&=A_{11}v_2,\\
    A_{22}v_1&=A_{12}v_2.
    \end{align*}
    On the other hand, the fact that $T=\theta^\sharp$ is $g$-orthogonal to the image of the Nijenhuis tensor implies the following:
    \begin{align*}
    A_{11}v_1+A_{21}v_2&=-ab_1,\\
    A_{12}v_1+A_{22}v_2&=-ab_2.
    \end{align*}
    Moreover, the condition that $D\theta$ is $J$-anti-invariant implies that 
    \begin{align*}
    a&=0,\\
    A_{22}v_1&=A_{21}v_2,\\
    A_{12}v_1&=A_{11}v_2.
    \end{align*} 
    
    We obtain then the following theorem
    \begin{thm}
    Let $\mathfrak{g}$ be a $4$-dimensional unimodular almost abelian Lie algebra equipped with a pluricanonical LCAK non-symplectic structure $(g, J, F, \theta)$. Then, $\mathfrak{g}$ is isomorphic to one of the following Lie algebras:
    \begin{enumerate}[label=\arabic*)]
    \item $\mathcal{A}_{4,1}:\, [e_2,e_4]=e_1,\quad [e_3,e_4]=e_2,$\\
    \item $\mathcal{A}_{3,6}\oplus \mathcal{A}_1:\,[e_1,e_3]=-e_2,\quad [e_2,e_3]=e_1$,\\
     \item $\mathcal{A}_{3,4}\oplus \mathcal{A}_1:\,[e_1,e_3]=e_1,\quad [e_2,e_3]=-e_2.$
\end{enumerate}
The simply connected Lie groups associated to these Lie algebras admit compact quotients. Moreover, if $J$ is integrable, then the Lie algebra is $\mathcal{A}_{4,1}$.
    \end{thm}
    \begin{proof}
From the above discussion, we deduce that $A=0$ and one of the $v_i$ is not zero. We have then three cases:

\begin{enumerate}
\item $b\cdot v= 0$, in particular when $J$ is integrable: the canonical Jordan form of  $\ad_{e_{4}|_\n}$ up to scaling is
$\begin{pmatrix} 
        0 & 1 &0\\
        0 & 0&1\\
        0 & 0&0\\
        \end{pmatrix},$ which corresponds to $\mathcal{A}_{4,1}$.

\item $b\cdot v>0$: the canonical Jordan form of  $\ad_{e_{4}|_\n}$ up to scaling is
$\begin{pmatrix} 
        0 & 0 &0\\
        0 & 1&0\\
        0 & 0&-1\\
        \end{pmatrix},$ which corresponds to $\mathcal{A}_{3,4}\oplus \mathcal{A}_1.$
 
\item $b\cdot v<0$: the canonical Jordan form of  $\ad_{e_{4}|_\n}$ up to scaling is
$\begin{pmatrix} 
        0 & 0 &0\\
        0 & 0&1\\
        0 & -1&0\\
        \end{pmatrix},$ which corresponds to $\mathcal{A}_{3,6}\oplus \mathcal{A}_1$.
\end{enumerate}
    \end{proof}

\bibliographystyle{abbrv}%\bibliographystyle{ieeetr}

\bibliography{LCK-balanced}

\end{document}